\documentclass{amsart}
\usepackage{fullpage}
\usepackage{stmaryrd}
\usepackage{amssymb}
\usepackage{mathrsfs}

\newtheorem{theorem}{Theorem}[section]
\newtheorem{lemma}[theorem]{Lemma}
\newtheorem{cor}[theorem]{Corollary}

\newtheorem*{thmA}{Theorem A}
\newtheorem*{thmB}{Theorem B}

\theoremstyle{definition}
\newtheorem{definition}[theorem]{Definition}

\theoremstyle{remark}

\numberwithin{equation}{section}

\newcommand{\FF}{{\mathbb{F}}}

\newcommand{\bC}{{\mathbf{C}}}

\newcommand{\bZ}{{\mathbf{Z}}}
\newcommand{\bF}{{\mathbf{F}}}
\newcommand{\KK}{{\mathbb{K}}}
\newcommand{\bN}{{\mathbf{N}}}

\newcommand{\Aut}{{\operatorname{Aut}}}
\newcommand{\Stab}{{\operatorname{Stab}}}
\newcommand{\Irr}{{\operatorname{Irr}}}

\newcommand{\Syl}{{\operatorname{Syl}}}
\newcommand{\SL}{{\operatorname{SL}}}
\newcommand{\Sp}{{\operatorname{Sp}}}
\newcommand{\fl}{{\operatorname{fl}}}
\newcommand{\GL}{{\operatorname{GL}}}
\newcommand{\Gal}{{\operatorname{Gal}}}

\newcommand{\SIRSp}{{\operatorname{SIRSp}}}
\newcommand{\SCRSp}{{\operatorname{SCRSp}}}

\newcommand{\Hom}{{\operatorname{Hom}}}
\newcommand{\SP}{{\operatorname{SP}}}
\newcommand{\EP}{{\operatorname{EP}}}
\newcommand{\NEP}{{\operatorname{NEP}}}

\newcommand{\NSP}{{\operatorname{NSP}}}

\newcommand{\dl}{{\operatorname{dl}}}
\newcommand{\cl}{{\operatorname{cl}}}
\newcommand{\Ker}{{\operatorname{Ker}}}
\newcommand{\Hall}{{\operatorname{Hall}}}
\newcommand{\stab}{{\operatorname{stab}}}
\newcommand{\GF}{\mbox{GF}}

\let\nor=\triangleleft

\begin{document}

\title{Blocks of small defect}

\author{YONG YANG}
\address{Department of Mathematics, University of Wisconsin-Parkside, 900 Wood Road, Kenosha, WI 53144, USA.}
\makeatletter
\email{yangy@uwp.edu}
\makeatother

\date{}

%\dedicatory{}

% at present the "communicated by" line appears only in ERA, PROC and JAG
%\commby{}

\begin{abstract}
Let $G$ be a finite solvable group, let $p$ be a prime such that $p \geq 5$ and $O_p(G)=1$, and we denote $|G|_p=p^n$, then $G$ contains a block of defect less than or equal to $\lfloor \frac {3n} 5 \rfloor$. Let $G$ be a finite solvable group and let $p^a$ be the largest power of $p$ dividing $\chi(1)$ for an irreducible character $\chi$ of $G$, we show that $|G:\bF(G)|_p \leq p^{3a}$ for $p \geq 5$. %Also, we obtain the best known bound for Huppert's $\rho-\sigma$ conjecture, $|\rho(G)| \leq 3\sigma(G)+2$ for $G$ solvable.
\end{abstract}

\maketitle
\maketitle
\Large
%%%%%%%%%%%%%%%%%%%%%%%%%%%%%%%%%%%%%%%%%%%%%%%%%%%%%%%%%%%%%%%%%%%%%%%%%
\section{Introduction} \label{sec:introduction8}
%A number of parallelisms between results on character degrees and conjugacy class sizes have been found since the 1980s. However, the reason for the existence of this parallelism is not clear.
It is a classical theme to study orbit structures of a finite group $G$ over a finite, faithful and completely reducible $G$-module $V$. One of the most important and natural questions about orbit structure is to establish the existence of an orbit of a certain size. For a long time, there has been a deep interest and need to examine the size of the largest possible orbits in linear group actions. In $2004$, A. Moret\'o and T.R. Wolf ~\cite{MOWOLF} studied the solvable linear groups and investigated the relation of the point stabilizers $\bC_G(v)$ and the Fitting subgroup series of $G$. They proved an important orbit theorem ~\cite[Theorem E]{MOWOLF} and applied it to obtain results showing that solvable groups have large character degrees and conjugacy classes. This orbit theorem and its consequences were used to obtain a number of results on several conjectures on class sizes, character degrees and zeros of characters.

In ~\cite{YY1}, Yang strengthened this result by showing the following. Suppose that $V$ is a faithful completely reducible $G$-module where $G$ is a finite solvable group, then there exists $v \in V$ and $K \triangleleft G$ such that $\bC_G(v) \subseteq K$, where the Fitting length of $K$ is less than or equal to $7$. An example ~\cite[Section 4]{YY1} was provided to show that the improvement is the best possible. Although one cannot say more in general because of this example, it is possible to show that there exits an element $v \in V$ such that the $p$-part of $\bC_G(v)$ is relatively small for all the primes $p \geq 5$. In this paper, we show the following orbit theorem of solvable linear groups. This theorem in some sense is the best possible as the semi-linear group shows us.

\begin{thmA} \label{generalcase}
  Let $\pi_0$ be the set of all the primes except $2$ and $3$. Let $G$ be a finite solvable group and let $V$ be a finite, faithful and completely reducible $G$-module(possibly of mixed characteristic). Then there exists $K \nor G$, $K \subseteq \bF_2(G)$ and there exist two $G$-orbits with representatives $v_a$, $v_b \in V$ such that for any $H \in \Hall_{\pi_0}(G)$, we have $\bC_H(v_a) \subseteq K$ and $\bC_H(v_b) \subseteq K$. The Hall $\pi_0$-subgroup of $K \bF(G)/\bF(G)$ and the Hall $\pi_0$-subgroup of $K \cap \bF(G)$ are abelian. %Furthermore, $\bC_{O_{\pi_0}(K \cap \bF(G))}(v_a) \cap \bC_{O_{\pi_0}(K \cap \bF(G))}(v_b)=1$. %??? $\bC_{H \cap \bF(G)}(v_a) \cap \bC_{H \cap \bF(G)}(v_b)=1$.
\end{thmA}

Theorem A can be used to study the following problems although they look different at the first glance.

Let $G$ be a finite group. Let $p$ be a prime and $|G|_p=p^n$. An irreducible ordinary character of $G$ is called $p$-defect $0$ if and only if its degree is divisible by $p^n$. It is an interesting problem to give necessary and sufficient conditions for the existence of $p$-blocks of defect zero. If a finite group $G$ has a character of $p$-defect $0$, then $O_p(G)=1$ ~\cite[Corollary 6.9]{WF}. Unfortunately, the converse is not true. Although the block of defect zero may not exist in general, one could try to find the smallest defect $d(B)$ of a block $B$ of $G$. One of the results along this line is proved by Espuelas and Navarro ~\cite[Theorem A]{AENA2}. Let $G$ be a (solvable) group of odd order such that $O_p(G) = 1$ and $|G|_p = p^n$, then $G$ contains a $p$-block $B$ such that $d(B) \leq \lfloor n/2 \rfloor$. The bound is best possible, as shown by an example in ~\cite{AENA2}. In the same paper, they raised the following question. If $G$ is a finite group with $O_p(G)=1$, $p \geq 5$, and $|G|_p=p^n$, does $G$ contain a block of defect less then $\lfloor \frac n 2 \rfloor$?

Using Theorem A, we prove the following result as a partial answer to this question. The bound we obtain here is pretty sharp since $\lfloor \frac n 2 \rfloor$ is the best one may get.
\begin{thmB}
Let $G$ be a finite solvable group, let $p$ be a prime such that $p \geq 5$ and $O_p(G)=1$, and we denote $|G|_p=p^n$. Then $G$ contains a $p$-block $B$ such that $d(B) \leq \lfloor \frac {3n} {5} \rfloor $.
\end{thmB}

Let $p^a$ denote the largest power of $p$ dividing $\chi(1)$ for an irreducible character $\chi$ of $G$. Moret\'o and Wolf ~\cite[Theorem A]{MOWOLF} proved that for $G$ solvable, there exists a product $\theta=\chi_1(1) \cdots \chi_t(1)$ of distinct irreducible characters $\chi_i$ such that $|G: \bF(G)|$ divides $\theta(1)$ and $t \leq 19$. This implies that $|G:\bF(G)|_p \leq p^{19a}$. They also suggest that a better bound $|G:\bF(G)|_p \leq p^{2a}$ might be true for all solvable groups. In fact, they believe ~\cite[Question 2.2]{MOWOLF} that for solvable groups one may find two irreducible characters $\chi_1$ and $\chi_2$ such that $|G: \bF(G)| \mid \chi_1(1) \chi_2(1)$. Although it is difficult to answer this question in general, we are able to prove a closely related result using the previous orbit theorem. As a corollary, we show that $|G:\bF(G)|_p \leq p^{3a}$ for $p \geq 5$.

%A direct application of Theorem A shows that for solvable group $G$, $|G:\bF(G)|_p \leq p^{3a}$ for $p \geq 5$.%and even something a bit stronger, namely the logarithm to the base of $p$ of the $p$-part of $|G: \bF(G)|$ is bounded in terms of $a$.

%If $P$ is a Sylow $p$-subgroup of a finite group $G$ it is reasonable to expect that the degrees of irreducible characters of $G$ somehow restrict those of $P$. If $p^a$ is the largest power of $p$ dividing $\chi(1)$ for an irreducible character $\chi$ of $G$, it is a consequence of work [10] on the height-zero conjecture that the derived length of $P$ is at most $2a+1$ for $G$ $p$-solvable. (This had previously been proved by Isaacs [16] for solvable groups.) However a $p$-group of derived length $2$ can have irreducible characters of arbitrarily large degree. If $b(P)$ denotes the largest degree of an irreducible character of $P$, then Conjecture 4 of Moreto¡ä [22] suggests $\log b(P)$ is bounded as a function of $a$. We have proven this for $G$ solvable and even something a bit stronger, namely the logarithm to the base of $p$ of the $p$-part of $|G: \bF(G)|$ is bounded in terms of $a$.

%Huppert's $\rho-\sigma$ conjecture is a problem of central importance in character theory; many people are devoted to the study of this problem.

The Huppert's $\rho-\sigma$ conjectures state that there is an irreducible character $\chi$ of $G$ and a conjugacy class $C$ of $G$ such that the degree of $\chi$ and $|C|$ are each divisible by many primes. For the character theoretic $\rho-\sigma$ problem, we define $\rho(G)$ be those primes that divide the degree of some irreducible character of $G$ and $\sigma(G)$ be the maximum number of primes dividing the degree of an irreducible character of $G$. Huppert conjectures that $|\rho(G)|$ can be bounded in terms of $\sigma(G)$, and if $G$ is solvable, then even $|\rho(G)| \leq 2\sigma(G)$. Up to now the best known bound for $G$ solvable is $|\rho(G)| \leq 3\sigma(G)+2$ by Manz of Wolf ~\cite[Theorems 1.4]{MAWOLF2}. Theorem A may be used to study Huppert's $\rho-\sigma$ conjectures and obtain the best known bound. %If we let $\rho^{*}(G)$ be those primes that divide the size of some conjugacy class of $G$ and let $\sigma^{*}(G)$ be the maximum number of distinct primes dividing the order of some conjugacy class of $G$, then Huppert conjectured that $|\rho^{*}(G)| \leq 3\sigma^{*}(G)$ for solvable groups. Up to now the best known bound for $G$ solvable is $|\rho^{*}(G)| \leq 4\sigma^{*}(G)$ by Zhang ~\cite{Zhang1}.
%This construction and the result of ~\cite{CASOLO} and related questions appear in ~\cite{HUPPERTC}. Casolo and Dolfi ~\cite{CADO} disproved the obvious conjecture by constructing solvable groups $G_n$ for which $|\rho^{*}(G_n)|/\sigma^{*}(G_n) \rightarrow 3$ as $n \rightarrow \infty$.

Theorem A also has connections to other questions about degrees of characters and lengths of conjugacy classes of solvable groups.

%It turns out that this orbit theorem will give a unified approach for the previous problems. As an application, one may show the following. Let $G$ be a finite solvable group with $O_p(G)=1$, $p \geq 5$, and $|G|_p=p^n$, then $G$ contains a block of defect less then $\lfloor \frac {3n} 5 \rfloor$. Let $p^a$ be the largest power of $p$ dividing $\chi(1)$ for an irreducible character $\chi$ of $G$. As another application, one can show that $|G:\bF(G)|_p \leq p^{3a}$ for $p \geq 5$. In observing the Huppert $\rho-\sigma$ conjectures, one may obtain the bound $|\rho^{*}(G)| \leq 4\sigma^{*}(G)+2$ and $|\rho(G)| \leq 3\sigma(G)+2$ for $G$ solvable.

%Actually, we show that certain invariants of G and its Fitting series divide the degrees of irreducible characters or products thereof, and these results have connections to other questions about degrees of characters of solvable groups. Summarizing, we have, where $\bF_i$ is the $i$th term in the Fitting series of $G$.

\section{Notation and Lemmas} \label{sec:Notation and Lemmas}

Notation:
\begin{enumerate}
\item Let $G$ be a finite group, let $S$ be a subset of $G$ and let $\pi$ be a set of different primes. For each prime $p$, we denote $\SP_p(S)=\{\langle x \rangle \ | \ o(x)=p, x \in S \}$ and $\EP_p(S)=\{x \ |\ o(x)=p, x \in S \}$. We denote $\SP(S)=\bigcup_{\mbox{p primes}}\SP_p(S)$, $\SP_{\pi}(S)=\bigcup_{p \in \pi}\SP_p(S)$, $\EP(S)=\bigcup_{\mbox{p primes}}\EP_p(S)$ and $\EP_{\pi}(S)=\bigcup_{p \in \pi}\EP_p(S)$. We denote $\NEP(S)=|\EP(S)|$, $\NEP_p(S)=|\EP_p(S)|$ and $\NEP_{\pi}(S)=|\EP_{\pi}(S)|$. We denote $\NSP(S)=|\SP(S)|$, $\NSP_p(S)=|\SP_p(S)|$ and $\NSP_{\pi}(S)=|\SP_{\pi}(S)|$.

\item Let $n$ be an even integer, $q$ a power of a prime. Let $V$ be a standard symplectic vector space of dimension $n$ of $\FF_q$. We use $\SCRSp(n,q)$ or $\SCRSp(V)$ to denote the set of all solvable subgroups of $\Sp(V)$ which acts completely reducibly on $V$. We use $\SIRSp(n,q)$ or $\SIRSp(V)$ to denote the set of all solvable subgroups of $\Sp(V)$ which acts irreducibly on $V$. Define $\SCRSp(n_1,q_1) \times \SCRSp(n_2,q_2)=\{H \times I | H \in \SCRSp(n_1,q_1)$ and $I \in \SCRSp(n_2,q_2)\}$.

\item If $V$ is a finite vector space of dimension $n$ over $\GF(q)$, where $q$ is a prime power, we denote by $\Gamma(q^n)=\Gamma(V)$ the semi-linear group of $V$, i.e.,
\[\Gamma(q^n)=\{x \mapsto ax^{\sigma}\ |\ x \in \GF(q^n), a \in \GF(q^n)^{\times}, \sigma \in \Gal(\GF(q^n)/\GF(q))\},\] and we define \[\Gamma_0(q^n)=\{x \mapsto ax\ | \ x \in \GF(q^n), a \in \GF(q^n)^{\times}\}.\]

\item We use $\bF(G)$ to denote the Fitting subgroup of $G$. Let $\bF_0(G) \leq \bF_1(G) \leq \bF_2(G) \leq \cdots \leq \bF_n(G)=G$ denote the ascending Fitting series, i.e. $\bF_0(G)=1$, $\bF_1(G)=\bF(G)$ and $\bF_{i+1}(G)/\bF_i(G)=\bF(G/\bF_i(G))$. $\bF_i(G)$ is the $i$th ascending Fitting subgroup of $G$.

\item Let $\pi_0$ be the set of all the primes except $2$ and $3$.
\end{enumerate}

\begin{definition} \label{defineEi}
Suppose that a finite solvable group $G$ acts faithfully, irreducibly and quasi-primitively on a finite vector space $V$. Let $\bF(G)$ be the Fitting subgroup of $G$ and $\bF(G)=\prod_i P_i$, $i=1, \dots, m$ where $P_i$ are normal $p_i$-subgroups of $G$ for different primes $p_i$. Let $Z_i = \Omega_1(\bZ(P_i))$. We define \[E_i=\left\{ \begin{array}{lll} \Omega_1(P_i) & \mbox{if $p_i$ is odd}; \\ \lbrack P_i,G,\cdots, G \rbrack & \mbox{if $p_i=2$ and $\lbrack P_i,G,\cdots, G \rbrack \neq 1$}; \\  Z_i & \mbox{otherwise}. \end{array} \right.\] By possible reordering we may assume that $E_i \neq Z_i$ for $i=1, \dots, s$, $0 \leq s \leq m$ and $E_i=Z_i$ for $i=s+1, \dots, m$. We define $E=\prod_{i=1}^s E_i$, $Z=\prod_{i=1}^s Z_i$ and we define $\overline{E}_i=E_i/Z_i$, $\bar{E}=E/Z$. Furthermore, we define $e_i=\sqrt {|E_i/Z_i|}$ for $i=1, \dots, s$ and $e=\sqrt{|E/Z|}$.
\end{definition}

Theorem ~\ref{Strofprimitive}, Lemma ~\ref{OrderG}, Lemma ~\ref{SpClifford} and Lemma ~\ref{OrderSpGamma} are proved in ~\cite{YY2} and ~\cite{YY4} but we include the proofs here for completeness.

\begin{theorem} \label{Strofprimitive}

Suppose that a finite solvable group $G$ acts faithfully, irreducibly and
quasi-primitively on an $n$-dimensional finite vector space $V$ over finite field $\FF$ of characteristic $r$. We use the notation in Definition ~\ref{defineEi}. Then every normal abelian subgroup of $G$ is cyclic and $G$ has normal subgroups $Z \leq U \leq F \leq A \leq G$ such that,
\begin{enumerate}
\item $F=EU$ is a central product where $Z=E \cap U=\bZ(E)$ and $\bC_G(F) \leq F$;
\item $F/U \cong E/Z$ is a direct sum of completely reducible $G/F$-modules;
\item $E_i$ is an extra-special $p_i$-group for $i=1,\dots,s$ and $e_i=p_i^{n_i}$ for some $n_i \geq 1$. Furthermore, $(e_i,e_j)=1$ when $i \neq j$ and $e=e_1 \dots e_s$ divides $n$, also $\gcd(r,e)=1$;
\item $A=\bC_G(U)$ and $G/A \lesssim \Aut(U)$, $A/F$ acts faithfully on $E/Z$;
\item $A/\bC_A(E_i/Z_i) \lesssim \Sp(2n_i,p_i)$;
\item $U$ is cyclic and acts fixed point freely on $W$ where $W$ is an irreducible submodule of $V_U$;% and we assume $|W|=|\FF|^l$;
\item $|V|=|W|^{eb}$ for some integer $b$;
\item $|G:A| \mid \dim(W)$. Assume $g \in G \backslash A$ and $o(g)=s$ where $s$ is a prime, then $|\bC_V(g)|=|W|^{\frac 1 s e b}$;
\item $G/A$ is cyclic.
\end{enumerate}
\end{theorem}
\begin{proof}
By \cite[Theorem 1.9]{manz/wolf} there exist $\tilde{E_i},T_i \triangleleft G$ and all the following hold,
\renewcommand*\theenumi{\roman{enumi}}
\renewcommand*\labelenumi{\theenumi)}
\begin{enumerate}
\item $P_i=\tilde{E_i}T_i$, $\tilde{E_i} \cap T_i=Z_i$ and $T_i=\bC_{P_i}(\tilde{E_i})$;
\item $\tilde{E_i}$ is extra-special or $\tilde{E_i}=Z_i$;
\item $\exp(\tilde{E_i})=p_i$ or $p_i=2$;
\item $T_i$ is cyclic, or $p_i=2$ and $T_i$ is dihedral, quaternion or semidihedral;
\item If $T_i$ is not cyclic, then there exists $U_i \nor G$ with $U_i$ cyclic, $U_i \leq T_i$, $|T_i:U_i|=2$ and $\bC_{T_i}(U_i)=U_i$;
\item If $\tilde{E_i}>Z_i$, then $E_i/Z_i=E_{i1}/Z_i \times \cdots \times E_{id}/Z_i$ for chief factors $E_{ik}/Z_i$ of $G$ and with $Z_i=\bZ(E_{ik})$ for each $k$ and $E_{ik} \leq \bC_G(E_{il})$ for $k \neq l$.
\end{enumerate}

We define $U_i=T_i$ if $p_i \neq 2$. We define $U=\prod_{i=1}^m U_i$, $T=\prod_{i=1}^m T_i$, $F=EU$ and $A=\bC_G(U)$.

If $p_i \neq 2$, then by (i),(ii),(iii) $\tilde{E_i}=\Omega_1(P_i)$ and therefore $\tilde{E_i}=E_i$.
If $p_i=2$ and assume $\tilde{E_i}>Z_i$, $\tilde{E_i}/Z_i=\prod_k E_{ik}/Z_i$ for chief factors $E_{ik}/Z_i$ of $G$ by (vi) and thus $E_{ik}=[E_{ik},G]$ and $\tilde{E_i} = [\tilde{E_i},G]$. By (v), $[T_i,G,\cdots, G]=1$. Thus $[P_i,G,G,\cdots, G]=\tilde{E_i}$ and therefore $\tilde{E_i}=E_i$.

The other results mainly follow from Corollary 1.10, 2.6 and Lemma 2.10 of \cite{manz/wolf}. Since $\bC_G(F)=\bC_G(EU) \leq \bC_G(E)=T$ and $\bC_T(U)=U$, we have $\bC_G(F) \leq F$. Since $A=\bC_G(U)$, $\bF(G) \cap A=\bC_{\bF(G)}(U)=EU=F$ and thus $A/F$ acts faithfully on $E/Z$.

Let $\KK$ be the algebraic closure of $\FF$, then $W \otimes_{\FF} \KK = W_1 \oplus W_2 \oplus \cdots \oplus W_m$, where the $W_i$ are Galois conjugate, non-isomorphic irreducible $U$-modules. In particular, each $W_i$ is faithful, $\dim_{\KK} W_i = 1$. Clearly $\bN_G(W_i) \geq \bC_G(U)$ for each $i$. Furthermore, $[\bN_G(W_i),U] \leq \bC_U(W_i)$ since $U$ is normal. Thus $\bN_G(W_i)=\bC_G(U)=A$. It follows that $G/A$ permutes the set $\{W_1, \dots , W_m \}$ in orbits of length $|G:A|$ and thus $|G:A| \mid \dim(W)$. Since $G/A$ permutes the $W_i$ fixed point freely, for all $g \in G \backslash A$ of order $s$ where $s$ is a prime, $|\bC_V(g)|=|W|^{\frac 1 s e b}$. This proves (8).

The fact that $G/A$ is cyclic is essentially proved in \cite[Section 20]{Suprunenko}. We give an argument here for completeness. $\FF U$ is a semi-simple algebra over $\FF$ and $\dim(\FF U)=|U|$. By Wedderburn's Theorem, there exist a finite number of idempotents $e_1, \dots, e_{\alpha}$ and $e_i \FF U$ is isomorphic to a full matrix algebra over some division algebra $D$ over $\FF$. Since $\FF U$ is commutative, the division algebra $D$ is actually a field and the dimension of the matrix is $1$. It follows that $e_i \FF U$ is a field for all $i=1,\dots,\alpha$. Set $K_i=e_i \FF U$, then $K_i, \dots, K_{\alpha}$ are fields and $\FF U=K_1 \oplus \cdots \oplus K_{\alpha}.$

Since $V$ is a quasi-primitive $G$-module, $V_U$ is a homogeneous $\FF U$-module. Then there exists some $j \in \{1,\dots,\alpha \}$ such that $e_j V=V$ and $e_l V=0$ for any $l \neq j$. Thus $V$ is a $K_j$ vector space. $K_j$ is a finite field extension of $\FF e_j$. Let $g \in G$, then $g K_j g^{-1}=K_j$ because $V$ is quasi-primitive. Define $\sigma_g: K_j \mapsto K_j$ as $\sigma_g(\beta)=g \beta g^{-1}$. $\sigma_g$ is a field automorphism of $K_j$ and fixes every element of $\FF e_j$. Thus $\sigma_g \in \Gal(K_j/\FF e_j)$. We claim that there is a natural map $\varphi$ between $G \mapsto \Gal(K_j/\FF e_j)$ defined by $\varphi(g)=\sigma_g$. Clearly $\varphi$ is a group homomorphism and $\bC_G(U) \subseteq \ker(\varphi)$. Suppose that $g \in G$, $g \not\in \bC_G(U)$ and let $u_0 \in U$ be a generator for $U$. Then $u_0 \neq gu_0g^{-1}$ and the action of $u_0$ on $V$ is different than the action of $gu_0g^{-1}$ on $V$ since the action of $G$ on $V$ is faithful. Thus the action of $e_j u_0$ on $V$ is different than the action of $e_j gu_0g^{-1}$ on $V$ and we know that $\sigma_g(e_j u_0) \neq e_j u_0$. Thus we have $\ker(\varphi)=\bC_G(U)=A$. This proves (9).
\end{proof}

\begin{lemma} \label{OrderG}
Suppose that a finite solvable group $G$ acts faithfully, irreducibly and
quasi-primitively on a finite vector space $V$. Using the notation in Theorem ~\ref{Strofprimitive}, we have $|G| \mid \dim(W) \cdot |A/F| \cdot e^2 \cdot (|W|-1)$.
\end{lemma}
\begin{proof}
By Theorem ~\ref{Strofprimitive}, $|G|= |G/A||A/F||F|$ and $|F|=|E/Z||U|$. Since $|G/A| \mid \dim(W)$, $|E/Z|=e^2$ and $|U| \mid (|W|-1)$, we have $|G| \leq \dim(W) \cdot |A/F| \cdot e^2 \cdot (|W|-1)$.
\end{proof}

\begin{lemma} \label{roughestimate}
Suppose that a finite solvable group $G$ acts faithfully and quasi-primitively on a finite vector space $V$ over the field $\FF$. Let $g \in \EP_s(G)$ where $s$ is a prime and we use the notation in Theorem ~\ref{Strofprimitive}.
\begin{enumerate}
\item If $g \in F$ then $|\bC_V(g)| \leq |W|^{\frac 1 2 eb}$.
%\item If $g \in A \backslash F$ then $|\bC_V(g)| \leq |W|^{\lfloor \frac 3 4 e \rfloor b}$.
\item If $g \in A \backslash F$, $s \geq 5$ and $s \nmid |E|$, then $|\bC_V(g)| \leq |W|^{\lfloor \frac 1 3 e \rfloor b}$.
%\item If $g \in A \backslash F$, $s=3$ and $s \nmid |E|$, then $|\bC_V(g)| \leq |W|^{\lfloor \frac 1 2 e \rfloor b}$.
%\item If $g \in A \backslash F$, $s=2$ and $s \nmid |E|$, then $|\bC_V(g)| \leq |W|^{\lfloor \frac 2 3 e \rfloor b}$.

\item If $g \in G \backslash A$ then $|\bC_V(g)| \leq |W|^{\frac 1 s eb}$.
\end{enumerate}
\end{lemma}
\begin{proof}
It is proved in \cite[Proposition 4.10]{manz/wolf} that for $g \in \bF(G)$, $|\bC_V(g)| \leq |W|^{\frac 1 2 eb}=|V|^{1/2}$. Since $F \leq \bF(G)$, (1) follows.

Since $\bC_G(F) \leq F$ and $g \not \in F$, $[g,F] \neq 1$. Since $g \in A=\bC_G(U)$ and $F=EU$, $[g,E] \neq 1$. Since $s \nmid |E|$, there exists a $g$-invariant $q$-subgroup $Q \leq E$, for some prime $q \neq s$ such that $Q$ is extra-special, $[Q,g]=Q$, $[\bZ(Q),g]=1$, $\bZ(Q) \triangleleft G$ and the action of $g$ on $Q/\bZ(Q)$ is fixed-point free. Let $\KK$ be a splitting field for $\langle g \rangle Q$ which is a finite extension of $\FF$ and set $V_{\KK}=V \otimes_{\FF} {\KK}$. Since $\dim_{\KK}(\bC_{V_{\KK}}(g))=\dim_{\FF}(\bC_V(g))$, we may consider $V_{\KK}$ instead of $V$. Let $0=V_0 \subset V_1 \subset \dots \subset V_l=V_{\KK}$ be a $\langle g \rangle Q$-composition series for $V_{\KK}$ with quotient $\overline{V}_j=V_j/V_{j-1}$ for $j=1, \dots ,l$. Thus each $\overline{V}_j$ is an absolute irreducible $\langle g \rangle Q$ module. Since $V_{\KK}$ is obtained by tensoring a quasi-primitive module up to a splitting field, $V_{\KK}|_{\bZ(Q)}$ is a direct sum of Galois conjugate irreducible modules, $\bZ(Q)$ is faithful on every irreducible summand of $V_{\KK}|_{\bZ(Q)}$. By the Jordan-Holder Theorem, these are the only irreducibles that can occur in $\overline{V}_j|_{\bZ(Q)}$ and thus $\bZ(Q)$ acts faithfully on $\overline{V}_j$. Since all nontrivial normal subgroups of $\langle g \rangle Q$ contain $\bZ(Q)$, $\langle g \rangle Q$ is faithful on $\overline{V}_j$. Since $g$ centralize $\bZ(Q)$, $\overline{V}_j$ is also an irreducible $Q$ module. Let $|Q|=q^{2k+1}$, then by Schult's Theorem \cite[Theorem V.17.13]{Huppert1} or Hall-Higman Theorem \cite[Theorem IX.2.6]{Huppert2}, $\dim_{\KK}(\overline{V}_j)=q^k \mid e$ and $\dim_{\KK}(\bC_{\overline{V}_j}(g)) \leq \lfloor \beta \cdot q^k \rfloor$where,

\[\beta=\left\{ \begin{array}{ll} \frac 1 s ({\frac {q^k+s-1} {q^k}}) & \mbox{if $s \mid q^k-1$};\\ \frac 1 s ({\frac {q^k+1} {q^k}}) & \mbox{if $s \mid q^k+1$}. \end{array} \right.\]

Assume $s \mid q^k-1$ and let $q^k-1=st$ where $t \geq 1$ is an integer, then $\beta=\frac {t+1} {st+1}$. Thus $\beta \leq \frac 1 3$ when $s \geq 5$.

Assume $s \mid q^k+1$ and let $q^k+1=st$ where $t \geq 1$ is an integer, then $\beta=\frac {t} {st-1}$. Thus $\beta \leq \frac 1 3$ when $s \geq 5$.

Since $\dim_{\KK}(\bC_{V_{\KK}}(g)) \leq \sum_j \dim_{\KK}(\bC_{\overline{V}_j}(g))$, (2) holds.

(3) follows from Theorem ~\ref{Strofprimitive}(8).
\end{proof}

\begin{lemma} \label{basiccounting}
Assume $G$ satisfies Theorem ~\ref{Strofprimitive} and we adopt the notation in it. Let $p$ be a prime and $x \in \EP_p(A \backslash F)$ and assume $|\bC_{E/Z}(x)|=\prod_i {p_i}^{m_i}$. We have the following:
\begin{enumerate}
\item $\NEP_p(A \backslash F) \leq \NEP_p(A/F)|F|$.
\item $\NEP_p(A \backslash F) \leq \frac {\NEP_{p}(A/F)|F|} {\prod_{p_i \neq p} p_i}$.
%\item $\NEP_p(xF) \leq \prod_i M_i \cdot U_{\bar{p}}$ where \[M_i=\left\{ \begin{array}{lll} {p_i}^{2n_i} & \mbox{if $p=p_i \neq 2$}; \\ {p_i}^{2n_i-m_i} & \mbox{if $p \neq p_i$}; \\ 2^{m_i} & \mbox{if $p=p_i=2$}. \end{array} \right.\]
%\item Assume further $p=2$ and $x$ is a good element. Define $S=\{y| y \in \EP_2(xF)$ and $y$ is a good element$\}$, then $|S| \leq \prod_i M_i \cdot U_{\bar{2}}$ where
%\[M_i=\left\{ \begin{array}{lll} {p_i}^{2n_i-m_i} & \mbox{if $p_i \neq 2$}; \\ 2^{m_i} & \mbox{if $p_i=2$ and $n_i \geq m_i$}; \\ 2^{2n_i-m_i} & \mbox{if $p_i=2$ and $n_i<m_i$}. \end{array} \right.\]
\end{enumerate}
\end{lemma}
\begin{proof}
This follows from \cite[Lemma 2.7]{YY2}.
\end{proof}

\begin{lemma} \label{topcounting}
Assume that $A$ is a normal subgroup of $G$ and $G/A$ is cyclic. Then we have $\NSP_{\pi_0}(G \backslash A) \leq \lfloor \log_5(|G/A|) \rfloor \cdot |A|$.
\end{lemma}
\begin{proof}
Since $G/A$ is cyclic, \[\NSP_{\pi_0}(G \backslash A) \leq \sum_{p \in \pi_0, p \mid |G/A|} (p-1) \cdot |A|/(p-1) = \sum_{p \in \pi_0, p \mid |G/A|} |A| \leq \log_5(|G/A|) \rfloor \cdot |A|.\]
\end{proof}

\begin{lemma}  \label{SpClifford}
Let $V$ be a symplectic vector space of dimension $n$ with base field $\FF$ and $G \in \SIRSp(n,\FF)$. Assume further the action is not quasi-primitive and that $N \triangleleft G$ is maximal such that $V_N$ is not homogeneous. Let $V_N=V_1 \oplus \cdots \oplus V_t$ where $V_i$'s are the homogeneous $N$-modules and clearly $t \geq 2$. Then either all $V_i$ are non-singular or all are totally isotropic. In the first case, $\dim(V_i)$ is even, $G \lesssim H \wr S$ as linear groups where $H \in \SIRSp(V_1)$. In the second case $t=2$, $V_2$ is isomorphic to $V_1^*$ as an $N$-module, and we say that $V_N$ is a pair.
\end{lemma}
\begin{proof}
$V$ is a symplectic $G$-module with respect to the non-singular symplectic form $(\ ,\ )$. By \cite[Proposition 0.2]{manz/wolf}, $S=G/N$ faithfully and primitively permutes the homogeneous components of $V_N$. Set $I=\bN_G(V_1)$, by Clifford's Theorem, $V_1$ is an irreducible $I$-module. Since the form $(\ ,\ )$ is $G$-invariant, the subspace $\{v\in V_1\ |\ (v,v')=0$ for all $v' \in V_1\}$  is an $I$-submodule of $V_1$ and the form $(\ ,\ )$ is either totally isotropic or non-singular on $V_1$. Since $G$ transitively permutes the $V_i$, the $G$-invariant form $(\ ,\ )$ is simultaneously totally isotropic or non-singular on all the $V_i$. If the form is non-singular on each $V_j$, then let $H=\bN_G(V_1)/\bC_G(V_1)$ and we know $H \in \SIRSp(V_1)$, $G \lesssim H \wr S$ as linear groups. Hence, we assume that each $V_i$ is totally isotropic.

Let $j \in \{ 1,\ldots,t\}$. Then, set $V_j^{\perp}=\{ v \in V\ |\ (v_,v_j)=0$ for all $v_j \in V_j \}$. For $v \in V$, we consider the map $f_v \in V_j^*:=\Hom_{\FF}(V_j,\FF)$, defined by $f_v(v_j)=(v,v_j), v_j \in V_j$. Then $v \mapsto f_v, v \in V$, induces a $N$-isomorphism between $V/V_j^{\perp}$ and the dual space $V_j^*$. Since $V_N$ is completely reducible, there exists an $N$-module $U_j$ such that $V_N=V_j^{\perp} \oplus U_j$. Thus $U_j \cong V_j^*$ is homogeneous and $U_j=V_{\pi(j)}$ for a permutation $\pi \in S_t$. Then $\pi$ is an
involution in $S_t$ and the permutation action of $S$ commutes with $\pi$. Hence, $S$ acts on the orbits of $\pi$. Since $\pi$ is not the identity, and the action of $S$ is primitive, it follows that $\pi$ has a single orbit and $t = 2$.
\end{proof}

\begin{lemma}  \label{OrderSpGamma}
Let $V$ be a symplectic vector space of dimension $2n$ with base field $\FF$ and $G \in \SIRSp(2n,\FF)$, $|\FF|=p$ where $p$ is a prime. Assume $G$ acts irreducibly and quasi-primitively on $V$ and $e=1$, then we have the following:
\begin{enumerate}
\item $G \lesssim \Gamma(p^{2n})$. $G/U$ is cyclic and $|G/U| \mid 2n$.
\item $U \lesssim \Gamma_0(p^{2n})$ and $|U| \mid p^n+1$.
%\item If $p=2$, then $\NPC(G,2,n) \leq 2^n+1$, $\NPC(G,2,i)=0$ for all $0 \leq i \leq 2n-1$ and $i \neq n$ and $\NEP_2(G) \leq 2^n+1$.
\end{enumerate}
\end{lemma}
\begin{proof}
By \cite[Proposition 3.1(1)]{Turull}, $G$ may be identified with a subgroup of the semi-direct product of $\GF(p^{2n})^{\times}$ by $\Gal(\GF(p^{2n}):\GF(p))$ acting in a natural manner on $\GF(p^{2n})^+$. Also $G \cap \GF(p^{2n})^{\times}=U$ and $|G \cap \GF(p^{2n})^{\times}| \mid p^n+1$. Clearly $G/U$ is cyclic of order dividing $2n$. Now (1) and (2) hold.
\end{proof}
%(3). Consider $x \in \EP_2(G)$ and thus $x \in G \backslash \bF(G)$, $x=\sigma a$ where $\sigma \in \Gal(\GF(p^{2n}):\GF(p))$ is of order 2 and $a \in \GF(p^{2n})^*$. Since $x^2=\sigma a \sigma a=1$, $a^{2^n+1}=1$ and thus $\NEP_2(G) \leq 2^n+1$. Since $2 \nmid 2^{2n}-1$, $x$ cannot act fixed point freely on $V$ and thus $x$ is conjugate with $\sigma$, clearly $|\bC_V(x)|=2^{n}$.

\begin{lemma}  \label{SGLbound}
Let $V$ be a finite, faithful irreducible $G$-module and $G$ is solvable. Suppose $V$ is a vector space of dimension $n$ over the field $\FF$.
Let $(n,\FF)=(4,\FF_2)$, then $|G| \leq 6^2 \cdot 2$ and $\NEP_{\{2,3\}'}(G) \leq 4$. Also $G_{\pi_0}$ is abelian and $G_{\pi_0} \subseteq \bF(G)$.
\end{lemma}
\begin{proof}
$G$ satisfies one of the following:
    \begin{enumerate}
    \item $G \lesssim S_3 \wr S_2$. Clearly $|G| \mid 6^2 \cdot 2$ and $G_{\pi_0}=1$.
    \item $V$ is irreducible and quasi-primitive, $e=1$ and $G \lesssim \Gamma(2^4)$. $|G|\mid 60$, $G_{\pi_0} \lesssim Z_5$ and $G_{\pi_0} \subseteq \bF(G)$. Clearly $\NEP_{\{2,3\}'}(G) \leq 4$.
    \end{enumerate}
Hence the result holds in all cases.
\end{proof}

\begin{lemma} \label{BdSCRSp}
Let $n$ be an even integer and $V$ be a symplectic vector space of dimension $n$ of field $\FF$. Let $G \in \SCRSp(n,\FF)$.
\begin{enumerate}
\item Let $(n,\FF)=(2,\FF_2)$, then $G \lesssim S_3$ and $|G| \mid 6$.
\item Let $(n,\FF)=(4,\FF_2)$, then $|G| \leq 6^2 \cdot 2$ and $\NEP_{\pi_0}(G) \leq 4$. Also $G_{\pi_0}$ is abelian and $G_{\pi_0} \subseteq \bF(G)$.
\item Let $(n,\FF)=(6,\FF_2)$, then $|G| \leq 6^4$, $\NEP_{\pi_0}(G) \leq 6$. Also $G_{\pi_0}$ is abelian and $G_{\pi_0} \subseteq \bF(G)$. %If $G$ is not a $2,3$-group then $\NEP_3(G)\leq 14$.
\item Let $(n,\FF)=(8,\FF_2)$, then $|G| \leq 6^4 \cdot 24$ and $\NEP_{\pi_0}(G)\leq 24$.
\item Let $(n,\FF)=(2,\FF_3)$, then $G \lesssim \SL(2,3)$ and $|G| \mid 24$.
\item Let $(n,\FF)=(4,\FF_3)$, then $|G| \leq 24^2 \cdot 2$, $\NEP_5(G) \leq 64$ and $G$ has no elements with prime order $p \geq 7$. If $5 \mid |G|$, then $|G| \leq 320$.

\end{enumerate}
\end{lemma}
\begin{proof}
$V$ is a symplectic $G$-module with respect to the non-singular symplectic form $(\ ,\ )$. If $V$ is not irreducible, we may choose an irreducible submodule $W$ of $V$ of smallest possible dimension and set $\dim(W)=m$. Since the form $(\ ,\ )$ is $G$-invariant, the subspace $\{v\in W \ |\ (v,v')=0$ for all $v' \in W\}$ is a submodule of $W$ and the form $(\ ,\ )$ is either totally isotropic or non-singular on $W$. If the form $(\ ,\ )$ is totally isotropic on $W$, then set $W^{\perp}=\{ v \in V \ |\ (v_,w)=0$ for all $w \in W \}$. For $v \in V$, we consider the map $f_v \in W^* := \Hom_{\FF}(W, \FF)$, defined by $f_v(w)=(v,w), w \in W$. Then $v \mapsto f_v, v \in V$, induces a $G$-isomorphism between $V/W^{\perp}$ and the dual space $W^*$. Since $V$ is completely reducible, we may find an irreducible $G$-submodule $U \cong W^*$ such that the form is non-singular on $X=W \oplus U$. If the form $(\ ,\ )$ is totally isotropic on $W$ and $n>2m$, then $V=X \oplus X^{\perp}$ where $X^{\perp}$ is not trivial. If the form $(\ ,\ )$ is non-singular on $W$, then $V=W \oplus W^{\perp}$ where $W^{\perp}$ is not trivial. In both cases we may view $G \lesssim \SCRSp(V_1) \times \SCRSp(V_2)$ as linear groups where $\dim(V_1)$, $\dim(V_2) <n$. If the form $(\ ,\ )$ is totally isotropic on $W$ and $n=2m$, then $V=W \oplus U$ and the action of $G$ on $V$ is a pair. We use this fact in the following arguments.

 We prove these different cases one by one.
\begin{enumerate}
\item Let $(n,\FF)=(2,\FF_2)$. Then $G \lesssim S_3$ and $|G| \mid 6$.
\item Let $(n,\FF)=(4,\FF_2)$. Assume $V$ is irreducible, then the result follows from Lemma ~\ref{SGLbound}.
Assume $V$ is reducible, then $G \lesssim S_3 \times S_3$, $|G| \mid 6^2$ and $G_{\pi_0}=1$.

\item Let $(n,\FF)=(6,\FF_2)$. Assume $V$ is irreducible and not quasi-primitive, then by Lemma ~\ref{SpClifford} $G$ satisfies one of the following:
\begin{enumerate}
\item $t=2$ and $\dim(V_1)=3$, the action of $N$ on $V$ must be a pair. Thus $N \lesssim \Gamma(2^3)$ and $\bF(N) \cong Z_7$. $G_{\pi_0} \cong Z_7$, $|G| \leq 21 \cdot 2=42$ and $G_{\pi_0} \subseteq \bF(G)$.
\item $t=3$ and $G \lesssim S_3 \wr S_3$. Thus $|G| \mid 6^4$ and $G_{\pi_0}=1$.
\end{enumerate}
Assume $V$ is irreducible and quasi-primitive, then $G$ satisfies one of the following:
\begin{enumerate}
\item $e=1$ and $|G| \leq (2^3+1) \cdot 6=54$ by Lemma ~\ref{OrderSpGamma}. Clearly $G_{\pi_0}=1$.

\item $e=3$. By Theorem ~\ref{Strofprimitive}, $A/F \lesssim \SL(2,3)$, $|W|=2^2$ and $\dim(W) \mid 2$. Thus $|U|=3$ and $|G/A| \mid 2$. $|G| \mid 2 \cdot 24 \cdot 3^3$ by Lemma ~\ref{OrderG}. Clearly $G_{\pi_0}=1$.
\end{enumerate}

Assume $V$ is reducible, then $G$ satisfies one of the following:
\begin{enumerate}
\item $G \lesssim \GL(3,2) \times \GL(3,2)$ and the action of $G$ on $V$ is a pair. Thus $G \lesssim \Gamma(2^3)$, $G_{\pi_0} \cong Z_7$ and $G_{\pi_0} \subseteq \bF(G)$. Thus $|G|\leq 21$ and $\NEP_{\{2,3\}'}(G) \leq 6$.
\item $G \in \SCRSp(2,2) \times \SCRSp(4,2)$. Thus $|G|\leq 6^3 \cdot 2$, $\NEP_{\{2,3\}'}(G) \leq 4$, $G_{\pi_0}$ is abelian and $G_{\pi_0} \subseteq \bF(G)$ by (1) and (2).
\end{enumerate}

Hence the result holds in all cases.

\item Let $(n,\FF)=(8,\FF_2)$. Assume $V$ is irreducible and not quasi-primitive, then by Lemma ~\ref{SpClifford} $G$ satisfies one of the following:
\begin{enumerate}
\item $t=2$ and $G \lesssim H \wr S_2$ where $H$ is an irreducible subgroup of $\GL(4,2)$. By Lemma ~\ref{SGLbound}, $|G| \leq (6^2 \cdot 2)^2 \cdot 2$ and $\NEP_{\{2,3\}'}(G) \leq 5 \cdot 5-1=24$.
\item $t=4$ and $G \lesssim S_3 \wr S_4$. Clearly $|G| \mid 6^4 \cdot 24$ and $G_{\pi_0}=1$.
\end{enumerate}
 Assume $V$ is irreducible and quasi-primitive. Since $2 \nmid e \mid 8$, $e=1$. By Lemma ~\ref{OrderSpGamma}, $|G| \mid (2^4+1) \cdot 8=136$, $G_{\pi_0} \cong Z_{17}$ and $\NEP_{\{2,3\}'}(G) \leq 16$.

Assume $V$ is reducible, then $G$ satisfies one of the following:
\begin{enumerate}
\item $G \lesssim \GL(4,2) \times \GL(4,2)$ and the action of $G$ on $V$ is a pair. Thus by Lemma ~\ref{SGLbound}, $|G| \leq 6^2 \cdot 2$ and $\NEP_{\{2,3\}'}(G) \leq 4$.
 \item $G \in \SCRSp(4,2) \times \SCRSp(4,2)$. Thus $|G| \leq 6^4 \cdot 4$, $\NEP_{\{2,3\}'}(G) \leq 5 \cdot 5-1=24$ by (2).

 \item $G \in \SCRSp(2,2) \times \SCRSp(6,2)$. Thus $|G| \leq 6^5$, $\NEP_{\{2,3\}'}(G) \leq 1 \cdot 7-1=6$ by (1) and (3).

\end{enumerate}

Hence the result holds in all cases.

\item Let $(n,\FF)=(2,\FF_3)$, then $G \lesssim \Sp(2,3) \cong \SL(2,3)$ and $|G| \mid 24$.

\item Let $(n,\FF)=(4,\FF_3)$. Assume $V$ is irreducible and not quasi-primitive, then by Lemma ~\ref{SpClifford} $G$ satisfies one of the following:
\begin{enumerate}
\item $t=2$ and the form is totally isotropic on $V_1$, the action of $N$ on $V$ is a pair. $|N| \mid |\GL(2,3)|=48$ and $|G| \mid 96$.

\item $t=2$ and the form is non-singular on $V_1$, $G \lesssim \SL(2,3) \wr S_2$. Thus $|G| \mid 24^2 \cdot 2$ and $G_{\pi_0}=1$.
%\item $t=4$ and $G \lesssim \GL(1,3) \wr S_4$. Since $\dim(V_1)=1$, the action of $N$ on $V$ must be paired and by Lemma ~\ref{SpClifford}, $|G| \mid 8 \cdot 2^2=32$. $\NEP_3(G)=0$ and $\NPC(G,3,3)=0$.
\end{enumerate}
Assume further the action of $G$ on $V$ is quasi-primitive. It is well known that a maximal subgroup of $\Sp(4,3)$ is isomorphic to one of the five groups $M_1$,$M_2$,$M_3$,$M_4$ and $M_5$, where $M_1 \cong \SL(2,3) \wr S_2$ and $|M_2|=|M_3|=2^4 \cdot 3^4$, $|O_3(M_2)|=|O_3(M_3)|=3^3$, $M_4=2.S_6$, $M_5=(D_8 \Ydown Q_8).A_5$. We can hence assume $G$ is maximal among the solvable subgroups of $M_2, M_3, M_4$ or $M_5$. Assume $G$ is a subgroup of $M_2$ or $M_3$, then clearly $|G| \mid 48$. Assume $G$ is a subgroup of $M_4$, it is not hard to show that $|G| \mid 96$ or $|G| \mid 40$. Assume $G$ is a subgroup of $M_5$, then $G \cong (D_8 \Ydown Q_8).L$ where $L$ is $S_3$, $A_4$ or $F_{10}$ and thus we know $|G| \leq 384$ and $|G| \leq 320$ if $5 \mid |G|$. It is checked by GAP ~\cite{GAP} that for all $G$ quasi-primitive and $|G| \leq 384$, $\NEP_5(G) \leq 64$ and $G$ will have no elements with other prime order.

Assume $V$ is reducible, then $G$ satisfies one of the following:
\begin{enumerate}
\item $G \lesssim \GL(2,3) \times \GL(2,3)$ and the action of $G$ on $V$ is a pair. Thus $|G|\mid 48$ and $G_{\pi_0}=1$.
\item $G \lesssim \SL(2,3) \times \SL(2,3)$. Thus $|G| \mid 24^2$ and $G_{\pi_0}=1$.
\end{enumerate}
Hence the result holds in all cases.
\end{enumerate}
\end{proof}

\begin{lemma}\label{permutation23}
Let $G$ be a finite solvable group on a finite set $\Omega$. Then there exists a subset $\Delta \subseteq \Omega$ such that $\Stab_G(\Delta)$ is a $\{2,3\}$-group. Here, $\Delta$ can be chosen to have non-empty intersection with every orbit of $G$ on $\Omega$.
\end{lemma}
\begin{proof}
This is \cite[Corollary 5.7(a)]{manz/wolf}.
\end{proof}

\section{Orbit theorems} \label{sec:gtheorem}
In this section, we prove the key orbit theorem of this paper. The proof relies on some of my previous work. \cite{YY2} and \cite{YY3} give a complete list of $e$ such that a solvable quasi-primitive group $G$ will have regular orbits on $V$. Based on this, we have the following.
\begin{theorem} \label{thm1}
Suppose that a finite solvable group $G$ acts faithfully, irreducibly and quasi-primitively on a finite vector space $V$. By Theorem ~\ref{Strofprimitive}, $G$ will have a uniquely determined normal subgroup $E$ which is a direct product of extra-special $p$-groups for various $p$ and $e=\sqrt{|E/\bZ(E)|}$. Assume $e=5,6,7$ or $e \geq 10$ and $e \neq 16$, then $G$ will have at least two regular orbits on $V$.
\end{theorem}
\begin{proof}
This is \cite[Theorem 3.1]{YY2} and \cite[Theorem 3.1]{YY3}.
\end{proof}

\begin{theorem}\label{quasiprimitivecase}
Let $G$ be a finite solvable group and let $V$ be a finite, faithful, irreducible and quasi-primitive $G$-module. Then there exists a normal subgroup $K \subseteq \bF_2(G)$ and there exist two $G$-orbits with representatives $v_a$, $v_b \in V$ such that for any $H \in \Hall_{\pi_0}(G)$, we have $\bC_H(v_a)$ and $\bC_H(v_b) \subseteq K$. The Hall $\pi_0$-subgroup of $K \bF(G)/\bF(G)$ and the Hall $\pi_0$-subgroup of $K \cap \bF(G)$ are abelian. Furthermore, either $\bC_{O_{\pi_0}(K \cap \bF(G))}(v_a)=1$ or  $\bC_{O_{\pi_0}(K \cap \bF(G))}(v_b)=1$.

\end{theorem}
\begin{proof}
   We adopt the notation in Theorem ~\ref{Strofprimitive}.
   By Theorem \ref{thm1}, we may assume $e=1,2,3,4,8,9,16$. Since $e$ is not divisible by a prime $p \geq 5$, we may assume that $|G/A||A/F|$ is divisible by some prime $p \geq 5$. Otherwise the result is clear since all the elements of $\pi_0$-order that belong to $\bF(G)$ act fixed point freely on any nontrivial vectors of $V$.

   In order to show that $G$ has two $\pi_0$-regular orbits on $V$ it suffices to check that \[\left | \bigcup_{P \in \SP_{\pi_0}(G)}\bC_V(P) \right |+|G| < |V|.\] We will divide the set $\SP_{\pi_0}(G)$ into a union of sets $A_i$. Clearly $\left | \bigcup_{P \in \SP_{\pi_0}(G)}\bC_V(P) \right | \leq \sum_i \left | \bigcup_{P \in A_i} \bC_V(P) \right |$. We will find $\beta_i < e$ such that $|\bC_V(P)| \leq |W|^{\beta_i b}$ for all $P \in A_i$. We will find $a_i$ such that $|A_i| \leq a_i$. Since $|V|=|W|^{eb}$, it suffices to check that \[\sum_i {a_i \cdot |W|^{\beta_i b}/|W|^{eb}}+|G|/|W|^{eb}< 1. \] We call this inequality $\star$.\\

  Assume $e=16$. Define $A_1=\{ \langle x \rangle \ |\ x \in \EP_{\pi_0}(A \backslash F) \}$. Thus $|\bC_V(P)| \leq |W|^{5 b}$ for all $P \in A_1$ by Lemma ~\ref{roughestimate}(2) and we set $\beta_1=5$. $|A_1| \leq 24 \cdot 2^8 \cdot (|W|-1)/2/4 = a_1$ by Lemma ~\ref{basiccounting}(2) and Lemma ~\ref{BdSCRSp}(4). Define $A_2=\{ \langle x \rangle \ | \ x \in \EP_{\pi_0}(G \backslash A)\}$. Thus $|\bC_V(P)| \leq |W|^{8 b}$ for all $P \in A_2$ by Lemma ~\ref{roughestimate}(3) and we set $\beta_2=8$. $|A_2| \leq \lfloor \log_5(\dim(W)) \rfloor \cdot 24 \cdot 6^4 \cdot 2^8 \cdot (|W|-1) = a_2$ by Lemma ~\ref{topcounting}. $|G|\leq \dim(W) \cdot 24 \cdot 6^4 \cdot 2^8 \cdot (|W|-1)$. It is routine to check that $\star$ is satisfied.\\

  Assume $e=9$.

  Assume that $p \mid |G/A|$ for some $p \geq 5$. Thus $p \mid \dim(W)$. Since $p \mid \dim(W)$ and $3 \mid |W|-1$, $|W| \geq 4^p$. Define $A_1=\{ \langle x \rangle \ | \ x \in \EP_{\pi_0}(A \backslash F)\}$. Thus $|\bC_V(P)| \leq |W|^{3 b}$ by Lemma ~\ref{roughestimate}(2) for all $P \in A_1$ and we set $\beta_1=3$. $|A_1| \leq 64 \cdot 3^4 \cdot (|W|-1)/3/4 = a_1$ by Lemma ~\ref{basiccounting}(2) and Lemma ~\ref{BdSCRSp}(6). Define $A_2=\{ \langle x \rangle\ |\ x \in \EP_{\pi_0}(G \backslash A)\}$. Thus $|\bC_V(P)| \leq |W|^{\frac 9 5 b}$ for all $P \in A_2$ and we set $\beta_2=\frac 9 5$ by Lemma ~\ref{roughestimate}(3). $|A_2| \leq \lfloor \log_5(\dim(W))\rfloor \cdot 24^2 \cdot 2 \cdot 3^4 \cdot (|W|-1) = a_2$ by Lemma ~\ref{topcounting}. $|G| \leq \dim(W)\cdot 24^2 \cdot 2 \cdot 3^4 \cdot (|W|-1)$. It is routine to check that $\star$ is satisfied.

  Assume that $p \nmid |G/A|$ for any $p \geq 5$. Thus we may assume that $p \mid |A/F|$ for some $p \geq 5$. $|A/F| \leq 320$ by Lemma ~\ref{BdSCRSp}(6). Define $A_1=\{ \langle x \rangle\ |\ x \in \EP_{\pi_0}(A \backslash F)\}$. Thus $|\bC_V(P)| \leq |W|^{3 b}$ for all $P \in A_1$  by Lemma ~\ref{roughestimate}(2) and we set $\beta_1=3$. $|A_1| \leq 64 \cdot 3^4 \cdot (|W|-1)/3/4 = a_1$ by Lemma ~\ref{basiccounting}(2) and Lemma ~\ref{BdSCRSp}(6). $|G| \leq \dim(W) \cdot 320 \cdot 3^4 \cdot (|W|-1)$. It is routine to check that $\star$ is satisfied.\\

  Assume $e=8$. Since $(A/F)_{\pi_0}$ is abelian and $(A/F)_{\pi_0} \subseteq \bF(A/F)$, we may assume that $p \mid |G/A|$ for some $p \geq 5$. Thus $p \mid \dim(W)$. Since $p \mid \dim(W)$ and $2 \mid |W|-1$, $|W| \geq 3^p$. Define $A_1=\{ \langle x \rangle\ |\ x \in \EP_{\pi_0}(A \backslash F)\}$. Thus $|\bC_V(P)| \leq |W|^{2 b}$ for all $P \in A_1$ by Lemma ~\ref{roughestimate}(2) and we set $\beta_1=2$. $|A_1| \leq 6 \cdot 2^6 \cdot (|W|-1)/2/4 = a_1$ by Lemma ~\ref{basiccounting}(2) and Lemma ~\ref{BdSCRSp}(3). Define $A_2=\{ \langle x \rangle\ |\ x \in \EP_{\pi_0}(G \backslash A)\}$. Thus $|\bC_V(P)| \leq |W|^{\frac 8 5 b}$ for all $P \in A_2$ by Lemma ~\ref{roughestimate}(3) and we set $\beta_2=\frac 8 5$. $|A_2| \leq  \lfloor \log_5(\dim(W)) \rfloor  \cdot 6^4 \cdot 2^6 \cdot (|W|-1)  = a_2$ by Lemma ~\ref{topcounting}. $|G| \leq \dim W \cdot 6^4 \cdot 2^6 \cdot (|W|-1)$. It is routine to check that $\star$ is satisfied.\\

  Assume $e=4$. Since $(A/F)_{\pi_0}$ is abelian and $(A/F)_{\pi_0} \subseteq \bF(A/F)$, we may assume that $p \mid |G/A|$ for some $p \geq 5$. Thus $p \mid \dim(W)$. Since $p \mid \dim(W)$ and $2 \mid |W|-1$, $|W| \geq 3^p$. Define $A_1=\{ \langle x \rangle\ |\ x \in \EP_{\pi_0}(A \backslash F)\}$. $|\bC_V(P)| \leq |W|^{b}$ for all $P \in A_1$ by Lemma ~\ref{roughestimate}(2) and we set $\beta_1=1$. $|A_1| \leq 4 \cdot 2^4 \cdot (|W|-1)/2/4 = a_1$ by Lemma ~\ref{basiccounting}(2) and Lemma ~\ref{BdSCRSp}(2). Define $A_2=\{ \langle x \rangle\ |\ x \in \EP_{\pi_0}(G \backslash A)\}$. Thus $|\bC_V(P)| \leq |W|^{\frac 4 5 b}$ for all $P \in A_2$ by Lemma ~\ref{roughestimate}(3) and we set $\beta_2= \frac 4 5$. $|A_2| \leq \lfloor \log_5(\dim(W)) \rfloor \cdot 6^2 \cdot 2 \cdot 2^4 \cdot (|W|-1)= a_2$ by Lemma ~\ref{topcounting}. $|G| \leq \dim W \cdot 6^2 \cdot 2 \cdot 2^4 \cdot (|W|-1)$. It is routine to check that $\star$ is satisfied.\\

  Assume $e=3$. Since $(A/F)_{\pi_0}=1$, we may assume that $p \mid |G/A|$ for some $p \geq 5$. Thus $p \mid \dim(W)$. Since $p \mid \dim(W)$ and $3 \mid |W|-1$, $|W| \geq 4^p$. Define $A_1=\{ \langle x \rangle\ |\ x \in \EP_{\pi_0}(G \backslash A)\}$. Thus $|\bC_V(P)| \leq |W|^{\frac 3 5 b}$ for all $P \in A_1$ by Lemma ~\ref{roughestimate}(3) and we set $\beta_1=\frac 3 5$. $|A_1| \leq \lfloor \log_5(\dim(W)) \rfloor \cdot 24 \cdot 3^2 \cdot (|W|-1)=a_1$ by Lemma ~\ref{topcounting}. $|G| \leq \dim W \cdot 24 \cdot 9 \cdot (|W|-1)$. It is routine to check that $\star$ is satisfied.\\

  Assume $e=2$. Since $(A/F)_{\pi_0}=1$, we may assume that $p \mid |G/A|$ for some $p \geq 5$. Thus $p \mid \dim(W)$. Since $p \mid \dim(W)$ and $2 \mid |W|-1$, $|W| \geq 3^p$. Define $A_1=\{ \langle x \rangle\ |\ x \in \EP_{\pi_0}(G \backslash A)\}$. Thus $|\bC_V(P)| \leq |W|^{\frac 2 5 b}$ for all $P \in A_1$ by Lemma ~\ref{roughestimate}(3) and we set $\beta_1=\frac 2 5$. Since $\Aut(S_3) \cong S_3$ and $\Aut(Z_3) \cong Z_2$, all the elements of prime order $p \geq 5$ in $G \backslash A$ acts trivially on $A/F$. Since $G/A$ is cyclic, $\NSP_{\pi_0}(G/F) \leq \log_5(\dim(W))$. It is easy to see that $|A_1| \leq \lfloor \log_5(\dim(W)) \rfloor \cdot 2^2 \cdot (|W|-1)=a_1$. $|G| \leq \dim W \cdot 6 \cdot 4 \cdot (|W|-1)$. It is routine to check that $\star$ is satisfied.\\

  Assume $e=1$, we have $G \leq \Gamma(V)$. We know that $\fl(G) \leq 2$, $(\bF_2(G)/\bF(G))_{\pi_0}$ and $\bF(G)_{\pi_0}$ are abelian.
\end{proof}

We now restate Theorem A for convenience. This theorem may be viewed as the linear group analog of Lemma ~\ref{permutation23}.

\begin{thmA} %\label{generalcase}
  Let $G$ be a finite solvable group and let $V$ be a finite, faithful and completely reducible $G$-module(possibly of mixed characteristic). Then there exists $K \nor G$, $K \subseteq \bF_2(G)$ and there exist two $G$-orbits with representatives $v_a$, $v_b \in V$ such that for any $H \in \Hall_{\pi_0}(G)$, we have $\bC_H(v_a) \subseteq K$ and $\bC_H(v_b) \subseteq K$. The $\pi_0$-subgroup of $K \bF(G)/\bF(G)$ and the $\pi_0$-subgroup of $K \cap \bF(G)$ are abelian. Furthermore, $\bC_{O_{\pi_0}(K \cap \bF(G))}(v_a) \cap \bC_{O_{\pi_0}(K \cap \bF(G))}(v_b)=1$. %??? $\bC_{H \cap \bF(G)}(v_a) \cap \bC_{H \cap \bF(G)}(v_b)=1$.
\end{thmA}
\begin{proof}
We consider minimal counterexample on $|G|+\dim V$.

Step 1. $V$ is an irreducible $G$-module. Assume not, we have $V=V_1 \oplus V_2$ and each $V_i$ is a non-trivial $G$-module. Let $C_i=\bC_G(V_i)$ and $V_i$ is a faithful $G/C_i$-module. Let $G_i=G/C_i$ and we know that $G \lesssim G_1 \times G_2$. There exists $v_{ia}$, $v_{ib} \in V_i$ where $v_{ia}$, $v_{ib}$ are in different $G_i$ orbits and $K_i \nor G_i$ such that for any $H_i \in \Hall_{\pi_0}(G_i)$, $\bC_{H_i}(v_{ia}) \subseteq K_i \subseteq \bF_2(G_i)$ and $\bC_{H_i}(v_{ib}) \subseteq K_i \subseteq \bF_2(G_i)$. Also the $\pi_0$-subgroup of $K_i \bF(G_i)/\bF(G_i)$ and the $\pi_0$-subgroup of $K_i \cap \bF(G_i)$ are abelian and $\bC_{H_i \cap \bF(G)}(v_{ia}) \cap \bC_{H_i \cap \bF(G)}(v_{ib})=1$.

Let $v_a=v_{1a}+v_{2a}$, $v_b=v_{1b}+v_{2b}$ and $K= G \cap (K_1 \times K_2)$. Let $H \in \Hall_{\pi_0}(G)$ and $H_i=HC_i/C_i$. $\bC_H(v_a) \subseteq \bC_{H_1}(v_{1a}) \times \bC_{H_2}(v_{2a}) \subseteq K_1 \times K_2$. $\bC_H(v_b) \subseteq \bC_{H_1}(v_{1b}) \times \bC_{H_2}(v_{2b}) \subseteq K_1 \times K_2$. Clearly $K \bF(G)/\bF(G) \subseteq K_1 \bF(G_1)/\bF(G_1) \times K_2 \bF(G_2)/\bF(G_2)$ and $K \cap \bF(G) \subseteq K_1 \cap \bF(G_1) \times K_2 \cap \bF(G_2)$.%Since $(\bC_H(v_{1a}+v_{2a}) \cap \bF_2(G))\bF(G)/\bF(G) \subseteq (\bC_{H_1}(v_{1a}) \cap \bF_2(G_1))\bF(G_1)/\bF(G_1) \times (\bC_{H_2}(v_{2a}) \cap \bF_2(G_2))\bF(G_2)/\bF(G_2)$ and $(\bC_H(v_{1b}+v_{2b}) \cap \bF_2(G))\bF(G)/\bF(G) \subseteq (\bC_{H_1}(v_{1b}) \cap \bF_2(G_1))\bF(G_1)/\bF(G_1) \times (\bC_{H_2}(v_{2b}) \cap \bF_2(G_2))\bF(G_2)/\bF(G_2)$. Since $\bC_H(v_{1a}+v_{2a}) \cap \bF(G) \subseteq \bC_{H_1}(v_{1a}) \cap \bF(G_1) \times \bC_{H_2}(v_{2a}) \cap \bF(G_2)$ and $\bC_H(v_{1b}+v_{2b}) \cap \bF(G) \subseteq \bC_{H_1}(v_{1b}) \cap \bF(G_1) \times \bC_{H_2}(v_{2b}) \cap \bF(G_2)$.

$\bC_{O_{\pi_0}(K \cap \bF(G))}(v_{1a}+v_{2a}) \cap \bC_{O_{\pi_0}(K \cap \bF(G))}(v_{1b}+v_{2b}) \subseteq (\bC_{O_{\pi_0}(K_1 \cap \bF(G_1))}(v_{1a}) \cap \bC_{O_{\pi_0}(K_1 \cap \bF(G_1))}(v_{1b})) \times (\bC_{O_{\pi_0}(K_2 \cap \bF(G_2))}(v_{2a}) \cap \bC_{O_{\pi_0}(K_2 \cap \bF(G_2))}(v_{2b}))=1$.

Step 2. If $V$ is not quasi-primitive and there exists a normal subgroup $N$ of $G$ such that $V_N=V_1 \oplus \dots \oplus V_m$ for $m>1$ homogeneous components $V_i$ of $V_N$. If $N$ is maximal with this property, then $S=G/N$ primitively permutes the $V_i$. Also $V=V_1^G$, induced from $\bN_G(V_1)$. Let $L_1=\bN_G(V_1)/\bC_G(V_1)$, then $L_1$ acts faithfully and irreducibly on $V_1$ and $G$ is isomorphic to a subgroup of $L_1 \wr S$. Now, by induction, $L_1$ has at least two orbits of elements with representatives $v_1,u_1 \in V_1$ and $K_1 \nor L_1$ such that for $H_1 \in \Hall_{\pi_0}(\bC_{L_1}(v_1))$ or $H_1 \in \Hall_{\pi_0}(\bC_{L_1}(u_1))$ we have $H_1 \subseteq K_1 \subseteq \bF_2(L_1)$. Also the $\pi_0$-subgroup of $K_1 \bF(L_1)/\bF(L_1)$ and the $\pi_0$-subgroup of $K_1 \cap \bF(L_1)$ are abelian. Also $\bC_{O_{\pi_0}(K_1 \cap \bF(L_1))}(v_1) \cap \bC_{O_{\pi_0}(K_1 \cap \bF(L_1))}(u_1)=1$. By Lemma ~\ref{permutation23}, there exist two $G$-orbits with representatives $v_a$, $v_b \in V$ and $K \nor G$ such that for $H \in \Hall_{\pi_0}(\bC_G(v_a))$ or $H \in \Hall_{\pi_0}(\bC_G(v_b))$ we have $H \subseteq K \subseteq \bF_2(G)$. Also the $\pi_0$-subgroup of $K \bF(G)/\bF(G)$ and the $\pi_0$-subgroup of $K \cap \bF(G)$ are abelian. Furthermore, $\bC_{O_{\pi_0}(K \cap \bF(G))}(v_a) \cap \bC_{O_{\pi_0}(K \cap \bF(G))}(v_b)=1$.

Step 3. $V$ is quasi-primitive, the claim follows by Theorem ~\ref{quasiprimitivecase}. Final contradiction.
\end{proof}

%\begin{theorem} \label{generalcase}
%  Let $G$ be a finite solvable group and let $V$ be a finite, faithful and completely reducible $G$-module(possibly of mixed characteristic). Then there exists a normal subgroup $K \subseteq \bF_2(G)$ and there exist two $G$-orbits with representatives $v_a$, $v_b \in V$ such that for any $H \in \Hall_{\pi_0}(G)$, we have $\bC_H(v_a) \subseteq K$ and $\bC_H(v_b) \subseteq K$. The $\pi_0$-subgroup of $K \bF(G)/\bF(G)$ and the $\pi_0$-subgroup of $K \cap \bF(G)$ are abelian. Furthermore, $\bC_{O_{\pi_0}(K \cap \bF(G))}(v_a) \cap \bC_{O_{\pi_0}(K \cap \bF(G))}(v_b)=1$. %??? $\bC_{H \cap \bF(G)}(v_a) \cap \bC_{H \cap \bF(G)}(v_b)=1$.
%\end{theorem}
%\begin{proof}
%The argument is routine.%Set $V= V_1 \oplus \cdots \oplus V_m$ for irreducible $G$-modules $V_i$
%\end{proof}

\begin{theorem} \label{generalcasep}
  Let $G$ be a finite solvable group and let $V$ be a finite, faithful and completely reducible $G$-module(possibly of mixed characteristic). Let $p$ be a prime and $p \geq 5$. Then there exists $K \nor G$, $K \subseteq \bF_2(G)$ and there exists a $G$-orbit with representative $v \in V$ such that for any $P \in \Syl_{p}(G)$ we have $\bC_P(v) \subseteq K$. Also the $p$-subgroup of $K \bF(G)/\bF(G)$ and the $p$-subgroup of $K \cap \bF(G)$ are abelian. Furthermore, $|\bC_{O_p(K \cap \bF(G))}(v)| \leq |O_p(K \cap \bF(G))|^{1/2}$. %??? $\bC_{P \cap \bF(G)}(v_a) \cap \bC_{P \cap \bF(G)}(v_b)=1$.
\end{theorem}
\begin{proof}
%Let $P$ be a Sylow $p$-subgroup of $G$, then $P \subseteq H$ for some $H \in \Hall_{\pi_0}(G)$.

By Theorem A, there exist two $G$-orbits with representatives $v_a$, $v_b \in V$ and $K \nor G$ such that for $H \in \Hall_{\pi_0}(G)$, we have $\bC_H(v_a)\subseteq K \subseteq \bF_2(G)$ and $\bC_H(v_b) \subseteq K \subseteq \bF_2(G)$. The $\pi_0$-subgroups of $K \bF(G)/\bF(G)$ and $K \cap \bF(G)$ are abelian. Furthermore $\bC_{O_{\pi_0}(K \cap \bF(G))}(v_a) \cap \bC_{O_{\pi_0}(K \cap \bF(G))}(v_b)=1$. Let $P \in \Syl_{p}(G)$, then $P \subseteq H \in \Hall_{\pi_0}(G)$ for some $H$. Thus $\bC_P(v_a) \subseteq \bC_H(v_a)\subseteq K \subseteq \bF_2(G)$ and $\bC_P(v_b) \subseteq \bC_H(v_b)\subseteq K \subseteq \bF_2(G)$. Also the $p$-subgroups of $K \bF(G)/\bF(G)$ and $K \cap \bF(G)$ are abelian.

Let $P_1=O_p(K \cap \bF(G))$, then $\bC_{P_1}(v_a) \cap \bC_{P_1}(v_b)=1$. Since $|\bC_{P_1}(v_a)| \cdot |\bC_{P_1}(v_b)| = \frac {|\bC_{P_1}(v_a)| \cdot |\bC_{P_1}(v_b)|} {|\bC_{P_1}(v_a) \cap \bC_{P_1}(v_b)|}=|\bC_{P_1}(v_a)  \bC_{P_1}(v_b)| \leq |P_1|$. It follows that, either  $|\bC_{P_1}(v_a)| \leq \sqrt{|P_1|}$ or $|\bC_{P_1}(v_b)| \leq \sqrt{|P_1|}$.
\end{proof}

\section{Blocks of small defect} \label{sec:Blocks}
Let $G$ be a finite group. Let $p$ be a prime and $|G|_p=p^n$. An irreducible ordinary character of $G$ is called $p$-defect $0$ if and only if its degree is divisible by $p^n$. By ~\cite[Theorem 4.18]{WF}, $G$ has a character of $p$-defect $0$ if and only if $G$ has a $p$-block of defect $0$. An important question in the modular representation theory of finite groups is to find the group-theoretic conditions for the existence of characters of $p$-defect $0$ in a finite group. It is an interesting problem to give necessary and sufficient conditions for the existence of $p$-blocks of defect zero. If a finite group $G$ has a character of $p$-defect $0$, then $O_p(G)=1$ ~\cite[Corollary 6.9]{WF}. Unfortunately, the converse is not true. Zhang ~\cite{Zhang} and Hiroshi ~\cite{Hiroshi1, Hiroshi2} provided various sufficient conditions where a finite group $G$ has a block of defect zero.

%Brauer¡¯s Problem 19, one of many conjectures and problems posed in [3], asks for a description of the number of defect zero p-blocks for a finite group in terms of its invariants. In [20], Robinson solved this problem; however it is difficult to determine his invariants for many groups.

%From the theory of modular representations of finite groups, we know [10,6.1.18] that an ordinary irreducible representation of a finite group $G$ is $p$-modularly irreducible and has defect zero if and only if the power of $p$ dividing the degree of the representation is equal to the power of $p$ dividing $|G|$.

%Results on regular orbit theorems have associated conditions for the existence of blocks of defect zero ~\cite{AE11}. For groups of odd order, with $O_p(G)=1$, the exceptions are essentially nilpotent by supersolvable as ~\cite[Theorem 1]{AE11} shows us.

Although the block of defect zero may not exist in general, one could try to find the smallest defect $d(B)$ of a block $B$ of $G$. One of the results along this line is given by ~\cite[Theorem A]{AENA2}. In ~\cite{AENA2}, Espuelas and Navarro bounded the smallest defect $d(B)$ of a block $B$ of $G$ using the $p$-part of $G$. Using an orbit theorem ~\cite[Theorem 3.1]{AE1} of solvable linear groups of odd order, they showed the following result.  Let $G$ be a (solvable) group of odd order such that $O_p(G) = 1$ and $|G|_p = p^n$, then $G$ contains a $p$-block $B$ such that $d(B) \leq \lfloor n/2 \rfloor$. The bound is best possible, as shown by an example in ~\cite{AENA2}.

It is not true in general that there exists a block $B$ with $d(B) \leq \lfloor n/2 \rfloor$, as $G = A_7 (p = 2)$ shows us. However, the counterexamples where only found for $p=2$ and $p=3$. By work of Michler and Willems ~\cite{Michler,Willems} every simple group except possibly the alternating group has a block of defect zero for $p \geq 5$. The alternating group case was settled by Granville and Ono in ~\cite{GranvilleOno} using number theory. In fact, they proved the $t$-core partition conjecture and the most difficult case was handled by modular forms. Based on this, the following question raised by Espuelas and Navarro ~\cite{AENA2} seems to be natural. If $G$ is a finite group with $O_p(G)=1$, $p \geq 5$, and $|G|_p=p^n$, does $G$ contain a block of defect less then $\lfloor \frac n 2 \rfloor$?

In this section, we study this question and show that for solvable group $G$, $O_p(G)=1$ and $p \geq 5$, $G$ contains a block of defect less than or equal to $\lfloor \frac {3n} 5 \rfloor$. The proof relies on the previous orbit theorem (Theorem ~\ref{generalcasep}). The bound we obtain here is pretty sharp since $\lfloor \frac n 2 \rfloor$ is the best one may get. We restate Theorem B for convenience.

\begin{thmB} \label{sdefect}
Let $G$ be a finite solvable group such that $O_p(G) = 1$ for $p \geq 5$ and let $|G|_p = p^n$. Then $G$ contains a $p$-block $B$ such that $d(B) \leq \lfloor \frac {3n} {5} \rfloor $.
\end{thmB}
\begin{proof}
Induction on $|G|$. Consider $\tilde G = G/\Phi(G)$. As $\bF(G/\Phi(G)) = \bF(G)/\Phi(G)$, we have that $O_p(\tilde G) = 1$ and $|\tilde G|_p = |G|_p$. If $\Phi(G) \neq 1$, then the result is true for $\tilde G$. Let $\tilde B$ be a $p$-block of $\tilde G$ such that $d(\tilde B) < \lfloor \frac {3n} 5 \rfloor$. By ~\cite[Lemma V.4.3]{WF}, there exists a $p$-block $B$ of $G$ such that $d(B) = d(\tilde B)$. Hence we may assume that $\Phi(G)= 1$.

%Now, $V = \Irr(\bF(G))$ is a faithful and completely reducible $G/\bF(G)$-module (over different fields, possibly). Put $V = V_1 \oplus \cdots \oplus V_t$, where each $V_i$ is an irreducible $G$-module. Define $K_i = \bC_G(V_i)$, $G_i = G/K_i$ and use the bar convention. By the Lemma above, there exists a normal subgroup $H_i$ of $G$ containing $K_i$ and an element $\lambda_i \in V_i$ such that $C_{\bar{G}}(\chi_i)^{\#} \subseteq \bar{H_i} - \bF(\bar{H_i})$. Furthermore, $H_i/\bF(H_i)$ is abelian. Consider $\lambda = \lambda_i \times \cdots \times \lambda_t$ and put $C = \bC_G(\lambda)$. We may view $G/\bF(G)$ as a subgroup of $\overline{G}_1 \times \cdots \times \overline{G}_t$. This shows that $C \subseteq \bF_3(G)$ and $C \cap \bF_2(G) = \bF(G)$, where, as usual, $\bF_1(G) = \bF(G)$, $\bF_i(G)/\bF_{i-1}(G) = \bF(G/\bF_{i-1}(G))$.
%We consider separately two cases.

Now, $V = \Irr(\bF(G))$ is a faithful and completely reducible $\bar G=G/\bF(G)$-module (over different fields, possibly). Put $V = V_1 \oplus \cdots \oplus V_t$, where each $V_i$ is an irreducible $\bar G$-module. Define $K_i = \bC_{\bar G}(V_i)$ and $G_i = \bar G/K_i$. By Theorem ~\ref{generalcasep}, there exists a normal subgroup $K \subseteq \bF_2(\bar G)$ and there exists a $\bar G$-orbit with representatives $\lambda \in V$ such that for any $P \in \Syl_{p}(\bar G)$ we have $\bC_P(\lambda) \subseteq K$. Also the $p$-subgroup of $K \bF(\bar G)/\bF(\bar G)$ and the $p$-subgroup of $K \cap \bF(\bar G)$ are abelian. Furthermore, $|\bC_{O_p(K \cap \bF(\bar G))}(\lambda)| \leq |O_p(K \cap \bF(\bar G))|^{1/2}$. Let $p^n=|\bar G|_p$, $p^{n_1}=|K \cap \bF(\bar G)|_p$, $p^{n_2}=|K \bF(\bar G) :\bF(\bar G)|_p$ and $p^{n_3}=|\bar G:K|_p$. Clearly $n=n_1+n_2+n_3$.

%(1) $|C|_p < |G|_p^{1/2}$.
Take $\chi \in \Irr(G)$ lying over $\lambda$ and let $B$ be the $p$-block of $G$ containing $\chi$. As $\bF(G)$ is a $p'$-group, \cite[Lemma V.2.3]{WF} shows that every irreducible character $\psi$ in $B$ has $\lambda$ as an irreducible constituent. Now $\psi(1)_p \geq |\bar G: \bC_{\bar G}(\lambda)|_p  \cdot |K \cap \bF(\bar G)|_p^{1/2}\geq p^{n_3} \cdot p^{n_1/2}$ by Theorem ~\ref{generalcasep}.

Let $P/\bF(\bar G)$ be the Sylow $p$-subgroup of $K \bF(\bar G)/\bF(\bar G)$ and let $P$ be the preimage of it. Let $Y =O_{p'}(\bF(\bar G))$, observe that $W = \Irr(Y/\Phi(Y))$ is a faithful and completely reducible $P/\bF(\bar G)$-module. Since $P/\bF(\bar G)$ is abelian, there exists $\mu \in W$ such that $C_P(\mu) = \bF(\bar G)$. We may view $\mu$ as a character of the preimage $X$ of $Y$ in $G$. Observe that $X$ is a $p'$-group. Take $\xi \in \Irr(G)$ lying over $\mu$. Now $\xi$ lies over an irreducible character $\phi$ of $P$ lying over $\mu$. Clearly, $\phi(1)_p \geq |P/\bF(\bar G)|=p^{n_2}$. As $P$ is normal in $G$, we have $\xi(1)_p \geq \phi(1)_p \geq p^{n_2}$. Let $B$ be the $p$-block of $G$ containing $\xi$. As $X$ is a $p'$-group, \cite[Lemma V.2.3]{WF} shows that every irreducible character $\delta$ in $B$ has $\mu$ as an irreducible constituent and $\delta(1)_p \geq p^{n_2}$.

%Let $P_1$ be the Sylow $p$-subgroup of $K \cap \bF(\bar G)$. $V$ is a faithful and completely reducible $P_1$-module. Since $P_1$ is abelian, there exists $\mu \in W$ such that $C_P(\mu) = \bF(\bar G)$. We may view $\mu$ as a character of the preimage $X$ of $Y$ in $G$. Observe that $X$ is a $p'$-group. Take $\xi \in \Irr(G)$ lying over $\mu$. Now $\xi$ lies over an irreducible character $\phi$ of $P$ lying over $\mu$. Clearly, $\phi(1)_p \geq |P/\bF(\bar G)|=p^{n_1}$. As $P$ is normal in $G$, we have $\xi(1)_p \geq \phi(1)_p \geq p^{n_2}$. Let $B$ be the $p$-block of $G$ containing $\xi$. As $X$ is a $p'$-group, \cite[Lemma V.2.3]{WF} shows that every irreducible character $\delta$ in $B$ has $\mu$ as an irreducible constituent and $\delta(1)_p \geq p^{n_1}$.

Let $P_1/\bF(G)$ be the Sylow $p$-subgroup of $K \cap \bF(\bar G)$ and let $P_1$ be the preimage of it. Since $P_1/\bF(G)$ is normal in $G/\bF(G)$, $V = \Irr(\bF(G))$ is a faithful and completely reducible $P_1/\bF(G)$-module. Since $P_1/\bF(G)$ is abelian, using a similarly argument as the previous paragraph, we may find a block $B$ such that every irreducible character $\varphi$ in $B$ satisfies $\varphi(1)_p \geq |K \cap \bF(\bar G)|_p=p^{n_1}$.

We know there is a block $B$ such that for every irreducible character $\alpha$ in $B$, $\alpha(1)_p \geq \max(p^{n_3} \cdot p^{n_1/2}, p^{n_2}, p^{n_1})$. It is not hard to see that $\alpha(1)_p \geq p^ {\frac {2n} 5} $ and thus $d(B) \leq \lfloor \frac {3n} {5} \rfloor$.
\end{proof}

Remark: Although the result for the solvable group case is satisfactory, the conjecture of Espuelas and Navarro for arbitrary finite groups is wide open.

\section{$p$ part of $|G:\bF(G)|$ and irreducible character degrees} \label{p part of G/F(G)}
%it is a consequence of work ~\cite{GLUCKWOLF} on the height-zero conjecture that the derived length of $P$ is at most $2a+1$ for $G$ $p$-solvable. However a $p$-group of derived length $2$ can have irreducible characters of arbitrarily large degree.
If $P$ is a Sylow $p$-subgroup of a finite group $G$ it is reasonable to expect that the degrees of irreducible characters of $G$ somehow restrict those of $P$. Let $p^a$ denote the largest power of $p$ dividing $\chi(1)$ for an irreducible character $\chi$ of $G$ and $b(P)$ denote the largest degree of an irreducible character of $P$. Conjecture $4$ of Moret\'o ~\cite{Moret1} suggested $\log b(P)$ is bounded as a function of $a$. Moret\'o and Wolf ~\cite{MOWOLF} have proven this for $G$ solvable and even something a bit stronger, namely the logarithm to the base of $p$ of the $p$-part of $|G: \bF(G)|$ is bounded in terms of $a$. In fact, they showed that $|G:\bF(G)|_p \leq p^{19a}$. Moret\'o and Wolf ~\cite{MOWOLF} also proved that $|G: \bF(G)|_p \leq p^{2a}$ for odd order groups, this can also be deduced from ~\cite{AENA2}. This bound is best possible, as shown by an example in ~\cite{AENA2}. It is possible that $p^a < b(P)$ at least when $p=2$, as shown by an example of Isaacs ~\cite[Example 5.1]{Moret1}. %(Isaacs has also provided solvable examples for $p=3$)

Moret\'o and Wolf ~\cite{MOWOLF} suggested that a better bound $|G:\bF(G)|_p \leq p^{2a}$ might be true for all solvable groups. In fact, they believe ~\cite[Question 2.2]{MOWOLF} that for solvable groups one may find two irreducible characters $\chi_1$ and $\chi_2$ such that $|G: \bF(G)| \mid \chi_1(1) \chi_2(1)$. Although it is difficult to answer this question in general, we are able to prove a closely related result using the previous orbit theorem. As a corollary, we show that $|G:\bF(G)|_p \leq p^{3a}$ for $p \geq 5$.

%Theorem A.
The following result is closely related to ~\cite[Theorem A]{MOWOLF}.
\begin{theorem}\label{pipartofGFG}
If $G$ is solvable, there exists a product $\theta=\chi_1 \chi_2 \chi_3$ of distinct irreducible characters $\chi_1$, $\chi_2$ and $\chi_3$ such that $|G:\bF(G)|_{\pi_0}$ divides $\theta(1)$.
\end{theorem}
\begin{proof}
By Theorem A, we may choose $\chi \in \Irr(G)$ and $K \nor G$ such that $\bF(G)$ is not in $\Ker \chi$ and $|G: K|_{\pi_0}$ divides $\chi(1)$. We can choose $\phi \in \Irr(G)$ such that $\bF(G)$ is in $\Ker \phi$ and $|K \bF_2(G): \bF_2(G)|_{\pi_0}$ divides $\phi(1)$. We can choose  $\mu \in \Irr(G)$ such that $\bF(G)$ is not in $\Ker \mu$ and $|K \cap \bF(G): \bF(G)|_{\pi_0}$ divides $\mu(1)$. %Then we have $\mu, \chi, \phi \in \Irr(\bF_2(G))$ such that $\mu^G$ is irreducible and $|\bF_3(G):\bF(G)|$ divides $\chi(1)\phi(1)$.

%Let $\gamma$ and $\delta$ in $\Irr(G)$ lie over $\chi$ and $\phi$.
Then $\phi$ is distinct since $\bF(G)$ is in $\Ker \phi$ but not in $\Ker \chi$ and $\Ker \mu$. If $\mu$ is $\chi$, the product $\theta=\chi \phi$ satisfies the conclusion. Else $\theta=\chi \phi \mu$ does.
\end{proof}

The following result is closely related to ~\cite[Theorem A']{MOWOLF}.
\begin{theorem}\label{pipartofGFGconj}
If $G$ is solvable, there exist conjugacy classes $C_1$, $C_2$ and $C_3$ such that $|G:\bF(G)|_{\pi_0}$ divides $|C_1| |C_2| |C_3|$.
\end{theorem}

%Corollary B.

The following result improves ~\cite[Corollary B]{MOWOLF} for $p \geq 5$.
\begin{cor}\label{chardegreebound}
Suppose that $p^{a+1}$ does not divide $\chi(1)$ for all $\chi \in \Irr(G)$ and let $P \in \Syl_p(G)$ where $p \geq 5$. If $G$ is solvable, then $|G: \bF(G)|_p\leq p^{3a}$, $b(P)\leq p^{4a}$ and $\dl(P) \leq \log_2 a + 7$.
\end{cor}
\begin{proof}
There exists a product $\theta=\chi_1 \chi_2 \chi_3$ of distinct irreducible characters $\chi_i$ such that $|G: \bF(G)|_{\pi_0}$ divides $\theta(1)$ by Theorem ~\ref{pipartofGFG} and so $|G:\bF(G)|_p \leq p^{3a}$. If $P \in \Syl_p(G)$, then $b(P) \leq |P: O_p(G)||b(O_p(G))|=|G:\bF(G)|_p |b(O_p(G))| \leq p^{3a}p^a=p^{4a}$.

Now, we want to prove the last part of the statement. By ~\cite[Theorem 12.26]{Isaacs/book} and the nilpotency of $P$, we have that $P$ has an abelian subgroup $B$ of index at most $b(P)^4$. By ~\cite[Theorem 5.1]{Podoski}, we deduce that $P$ has a normal abelian subgroup $A$ of index at most $|P:B|^2$. Thus, $|P:A| \leq |P:B| \leq b(P)^{8s}$, where $b(P)=p^s$. By ~\cite[Satz III.2.12]{Huppert1}, $\dl(P/A) \leq 1+\log_2(8s)$ and so $\dl(P) \leq 2+ \log_2(8s)=5+\log_2(s)$. Since $s$ is at most $4a$, the result follows.
\end{proof}

We now state the conjugacy analogs of Theorem ~\ref{chardegreebound}. Given a group $G$, we write $b^*(G)$ to denote the largest size of the conjugacy classes of $G$. The following result improves ~\cite[Corollary B']{MOWOLF} for $p \geq 5$.
\begin{cor}\label{conjugacybound}
Suppose that $p^{a+1}$ does not divide $|C|$ for all $C \in \cl(G)$ and let $P \in \Syl_p(G)$ where $p \geq 5$. If $G$ is solvable, then $|G: \bF(G)|_p \leq p^{3a}$, $b^*(P)\leq p^{4a}$ and $|P'| \leq p^{2a(4a+1)}$.
\end{cor}
\begin{proof}
The first statement follows directly from Theorem ~\ref{pipartofGFGconj}. Write $N=O_p(G)$. It is clear that $|N: \bC_N(x)|$ divides $|G: \bC_G(x)|$ for all $x \in G$. Thus, if we take $x \in P$ we have that
\[|\cl_P(x)|= |P: \bC_P(x)| \leq |P:N||N: \bC_N(x)| \leq p^{3a} p^a=p^{4a}\]
 Finally, to obtain the bounds for the order of $P'$ is suffices to apply a theorem of Vaughan-Lee \cite[Theorem VIII.9.12]{Huppert2}.
\end{proof}

\section{Huppert $\rho-\sigma$ conjectures} \label{sec:Huppert}
In this section we discuss Huppert's $\rho-\sigma$ conjectures.

If $n$ is a positive integer, we denote by $\pi(n)$ the set of all prime divisors of $n$. Let $\chi$ be an irreducible complex character of a group $G$, we denote by $\pi(\chi)$ the set of all prime divisors of the degree $\chi(1)$ of $\chi$. We define \[\sigma(G)=\max \{|\pi(\chi)|: \chi \in \Irr(G) \}\ and \ \rho(G)= \bigcup_{\chi \in \Irr(G)} \pi(\chi). \]
Thus $\rho(G)$ are those primes that divide the degree of some irreducible character of $G$ and $\sigma(G)$ is the maximum number of primes dividing the degree of an irreducible character of $G$. By Ito's theorem, $\rho(G)$ is precisely the set of all primes $p$ such that $G$ does not have a normal abelian Sylow $p$-subgroup.

Similarly, if $g \in G$, we denote by $\pi(g^G)$ the set of all prime divisors of $|G : \bC_G(g)|$, the size of the conjugacy class of $g \in G$. We define \[\sigma^{*}(G)=\max \{|\pi(g^G)|: g \in G \} \ and \ \rho^{*}(G)= \bigcup_{g \in G} \pi(g^G). \]
Thus $\rho^{*}(G)$ is the set of all prime divisors of the sizes of conjugacy classes of G. It is an elementary fact that $\rho^{*}(G)=\pi(G/\bZ(G))$. $\sigma^{*}(G)$ is the maximum number of distinct primes dividing the order of some conjugacy class of $G$.

A lot of research has been made on character degrees of finite groups since the eighties due to the interest of B. Huppert. One of the main problems that Huppert raised was his well-known $\rho-\sigma$ conjectures. The Huppert's $\rho-\sigma$ conjectures state that there is an irreducible character $\chi$ of $G$ and a conjugacy class $C$ of $G$ such that the degree of $\chi$ and $|C|$ are each divisible by many primes. Huppert's $\rho-\sigma$ conjecture is a problem of central importance in group theory and character theory; many people are devoted to the study of this problem.

For the character theoretic $\rho-\sigma$ problem,  Huppert conjectured that $|\rho(G)|$ can be bounded in terms of $\sigma(G)$, and if $G$ is solvable, then even $|\rho(G)| \leq 2\sigma(G)$. In ~\cite{Moret2} Moreto proved that, for any group $G$, $|\rho(G)| \leq 4\sigma(G)^2+6.5\sigma(G)+13$. This bound was improved to $|\rho(G)| \leq 7\sigma(G)$ by Casolo and Dolfi ~\cite[Theorem 1]{CADO2}. Up to now the best known bound for $G$ solvable is $|\rho(G)| \leq 3\sigma(G)+2$ and even $|\rho(G)| \leq 3\sigma(G)$ for $|G|$ odd by Manz and Wolf ~\cite[Theorems 1.4 and 1.5]{MAWOLF2}.%There are solvable groups $G_n$ and non-solvable groups $H_n$ for each $n \in \mathbb{N}$ for which $\sigma(G_n)=n=\sigma(H_n)$ while $|\rho(G_n)|=2n$ and $|\rho(H_n)|=2n+1$.

For the conjugacy class $\rho^*-\sigma^*$ problem, Huppert also conjectured that $|\rho^{*}(G)| \leq 2\sigma^{*}(G)$ for $G$ solvable. Casolo ~\cite{CASOLO} showed that $|\rho^{*}(G)| \leq 2\sigma^{*}(G)$ for a very large family of groups. But Casolo and Dolfi ~\cite{CADO} disproved the obvious conjecture by constructing solvable groups $G_n$ for which $|\rho^{*}(G_n)|/\sigma^{*}(G_n) \rightarrow 3$ as $n \rightarrow \infty$. In ~\cite{Moret2} Moreto proved that, for any group $G$, $|\rho(G)| \leq 3\sigma^*(G)^2+7.5\sigma^*(G)+11$. This bound was improved to $|\rho(G)| \leq 7\sigma(G)$ by Casolo and Dolfi ~\cite[Theorem 2]{CADO2}. Up to now the best known bound for $G$ solvable is $|\rho^{*}(G)| \leq 4\sigma^{*}(G)$ by Zhang ~\cite{Zhang1}. %We use $\pi_0(n)$ to denote the set of prime divisors of the integer $n$.This construction and the result of ~\cite{CASOLO} and related questions appear in ~\cite{HUPPERTC}. But Casolo and Dolfi ~\cite{CADO} do show $|\rho^{*}(G)| \leq 5\sigma^{*}(G)+1$ for all $G$ and $|\rho^{*}(G)| \leq 4\sigma^{*}(G)+1$ for $G$ solvable.

%What is more for $|G|$ odd, is that there exist three irreducible characters $\chi_i$ of $G$ such that each prime in $\rho(G)$ divides $\chi_i(1)$ or some $i$. This follows from Theorem D, since $G$ always has an irreducible character divisible by every prime $p$ for which $O_p(G)$ is non-abelian.

Theorem A yields linear bounds for arbitrary solvable groups in both versions of the problem. Thus, it provides a unified approach to the character-theoretic and the conjugacy class version of the $\rho-\sigma$ conjectures. The following theorem is about the character-theoretic version of the $\rho-\sigma$ conjectures.

\begin{theorem}\label{halfprimes}
Suppose that $M$ is a normal elementary abelian subgroup of the solvable group $G$. Assume that $M=\bC_G(M)$ is a completely reducible $G$-module (possibly of mixed characteristic). Set $V= \Irr(M)$ and write $V= V_1 \oplus \cdots \oplus V_m$ for irreducible $G$-modules $V_i$. For each $i$, write $V=Y_i^G$ for primitive modules $Y_i$. Then there exists $\Irr(G)$ whose degree is divisible by at least half the primes of $\pi_0(G/M)$.
\end{theorem}
\begin{proof}

We may write each $V_i$ as a direct sum of the $G$-conjugates of $Y_i$, $i=1,\dots, m$. Consequently, $V=X_1 \oplus \cdots \oplus X_n$ for subspaces $X_i$ of $V$ permuted by $G$ (not necessarily transitively) with $\{Y_1,\dots, Y_m\} \subseteq \{X_1,\dots, X_n\}$. Furthermore, if $N_i= \bN_G(X_i)$, $C_i=\bC_G(X_i)$ and $F_i/C_i=\bF(N_i/C_i)$, then $X_i$ is a primitive, faithful $N_i/C_i$-module. We denote by $\overline{N}_i=N_i/C_i$.% and $N_i/F_i$ is nilpotent.

Let $N = \bigcap_i N_i \nor G$ be the kernel of the permutation representation of $G$ on $\{X_1,\dots, X_n\}$. Since $\bigcap_i C_i=M$, we have $\bigcap_i F_i/M= \bF(N/M) \nor G/M$. Let $F=\bigcap_i F_i$, so that $F/M=\bF(N/M)$. %Observe that $F/M=\bF(N/M)$ and that $N/F$ is nilpotent. %Because $N/M/\bF(N/M)$ is nilpotent, a fairly standard argument yields the existence of $\mu \in \Irr(N/M)$ such that $\pi_0(\mu(1))=\pi_0(N/F)$(e.g., see Lemma 1.1 of [HM]).

By Lemma ~\ref{permutation23}, we may choose $\Delta \in \{X_1,\dots, X_n\}$ such that $\stab_G(\Delta)/N$ is a $\{2,3\}$-group. Furthermore, we can assume that $\Delta$ intersects each $G$-orbit non-trivially. Without loss of generality, $\Delta=\{X_1,\dots, X_l\}$ for some $\l \in \{1,\dots,n\}$.

%Thus, for $i= 1, \dots, l$, we may choose non-principal $\lambda_i \in X_i$ such that $\lambda_i$ is not centralized by a non-trivial Sylow $q$-subgroup of $F_i \cap N/C_i \cap N$.
Let $\Delta_{i1}=Y_i^G \cap \Delta$ and $\Delta_{i2}=Y_i^G \backslash \Delta_{i1}$ where $Y_i \in \{Y_1,\dots, Y_m\}$.

Thus, for $j \in \Delta_{i1}$, we may choose non-principle $\lambda_j=\lambda_{ia}^g \in X_j$ such that there exists a normal subgroup $\overline{K}_j \subseteq \bF_2(\overline{N}_j)$ and for any $H \in \Hall_{\pi_0}(\overline{N}_j)$, we have $\bC_H(\lambda_{j}) \subseteq \overline{K}_j$ by Theorem ~\ref{quasiprimitivecase}. The $\pi_0$-subgroups of $(\overline{K}_j  \cap \bF_2(\overline{N}_j))\bF(\overline{N}_j)/\bF(\overline{N}_j)$ are abelian. For $j \in \Delta_{i2}$, we may choose $\lambda_j=\lambda_{ib}^g \in X_j$ such that there exists a normal subgroup $\overline{K}_j \subseteq \bF_2(\overline{N}_j)$ and for any $H \in \Hall_{\pi_0}(\overline{N}_j)$, we have $\bC_H(\lambda_{j}) \subseteq \overline{K}_j$. The $\pi_0$-subgroups of $(\overline{K}_j  \cap \bF_2(\overline{N}_j))\bF(\overline{N}_j)/\bF(\overline{N}_j)$ are abelian. Here $\lambda_{ia}$ and $\lambda_{ib}$ belong to different $N_i$ orbits.

We define $\lambda=\lambda_1 \dots \lambda_n \in V$. %for non-principal $\lambda_i \in X_i$.

Finally suppose that $Q \in \Syl_q(G)$ for a prime $q \geq 5$, and $Q$ centralizes $\lambda$. Thus $Q \subseteq \stab_G(\Delta)$. But $\stab_G(\Delta)/N$ is a $\{2,3\}$-group. Thus $Q \subseteq N$ and we know that $Q \subseteq K=\bigcap K_j \subseteq \bF_2(N/M)$ and the $\pi_0$-subgroup of $KF/F$ is abelian.

Since $\lambda_i$ is non-principle for $i= 1, \dots, l$, $\lambda_i$ is not centralized by a non-trivial Sylow $q$-subgroup of $F_i \cap N/C_i \cap N$ by Theorem ~\ref{Strofprimitive}(6). Since $Q \cap F_i \in \Syl_q(F_i \cap N)$, we have that $q \nmid |F_i \cap N/C_i \cap N|$ for $i=1,\dots,l$. By our choice of $\Delta$, each $F_j/C_j (j=1,\dots,n)$ is conjugate to some $F_i/C_i$ with $i \in \{1,\dots,l\}$. Hence \[q \nmid |F_j \cap N/C_j \cap N|\] for all $j=1,\dots,n$. Since $\bigcap_i C_i=M$ and $\bigcap_i(F_i \cap N)=F$, we have that $q \nmid |F/M|$. We have seen above that $Q \subseteq K$ and so $q \nmid |G/K|$, Thus $|G: \bC_G(\lambda)|$ is divisible by every prime $p \geq 5$ in $\pi_0(G/K) \cup \pi_0(F/M)$. %For each $i$, $F_i \cap N/C_i \cap N$ is isomorphic to a normal nilpotent subgroup of $N_i/C_i$, and $N_i/C_i$ acts irreducibly on $X_i$.

Let $Z =N/M$, observe that $W = \Irr(\bF(Z)/\Phi(Z))$ is a faithful and completely reducible $Z/\bF(Z)$-module. Since the $\pi_0$-subgroup of $KF/F$ is abelian, there exits $\mu \in W$ such that $|KF/F|_{\pi_0} \mid \mu(1)$.

Now let
\[\beta \in \Irr(G|\mu)\ and\ \chi \in \Irr(G|\lambda). \]

By the last two paragraphs, $\beta(1)$ is divisible by every prime in $\pi_0(KF/F)$ and $\chi(1)$ is divisible by every prime in $\pi_0(G/K) \cup \pi_0(F/M)$. The conclusion of the lemma is met with $\theta = \beta$ or $\theta=\chi$.
\end{proof}

The following theorem obtain the known bound of the character version of the Huppert's $\rho-\sigma$ conjecture for $G$ solvable. The bound we obtain here is the same as what Manz and Wolf obtained in ~\cite[Theorems 1.4]{MAWOLF2}.
\begin{theorem}
Let $\rho(G)$ to be those primes that divide the degree of some irreducible character of $G$, i.e., $p \in \rho(G)$ if and only if $p$ divides $|G: \bF(G)|$ or $O_p(G)$ is non-abelian. Let $\sigma(G)$ denote the maximum number of primes dividing the degree of an irreducible character of $G$. If $G$ is solvable, then $|\rho(G)| \leq 3\sigma(G)+2$.
\end{theorem}
\begin{proof}
Let $\mathscr{R}=\{r$ prime $|\ O_r(G) \in \Syl_r(G)$ and $O_r(G)$ is non-abelian$\}$ and $F=\bF(G)$. Certainly $\rho(G) \subseteq \pi(G/F) \cup \mathscr{R}$ and by Ito's Theorem \cite[12.33]{Isaacs/book}, equality holds.

$\bF(G)/\Phi(G)$ is a faithful completely reducible $G/F$-module. Applying Theorem ~\ref{halfprimes} with $G/\Phi(G)$ and $\bF(G)/\Phi(G)$ in the role of $G$ and $M$, there exists $\chi \in \Irr(G)$ with $|\pi_0(\chi(1))| \geq |\pi_0(G/F)|/2$. Hence $\sigma (G) \geq |\pi_0(G/F)|/2$. Now $\prod_{r \in \mathscr{R}} O_r(G) \nor G$ and each $O_r(G)$ is non-abelian. Thus there exists $\eta \in \Irr(G)$ such that $\mathscr{R} \subseteq \pi(\eta(1)))$. Since $\sigma(G) \geq \max \{|\mathscr{R}|, |\pi_0(G/F)|/2 \}$ and since $\rho(G) = \pi(G/F) \cup \mathscr{R} \subseteq \pi_0(G/F) \cup \mathscr{R} \cup \{2, 3\}$, the result follows.
\end{proof}

Using the same argument, one may get a similar result for the conjugacy class version of the Huppert $\rho-\sigma$ conjectures (i.e. If $G$ is solvable, then $|\rho^*(G)| \leq 4\sigma^*(G)+2$).\\

%\section{Acknowledgement} \label{sec:Acknowledgement}
%I wish to thank Alexandre Turull for his constant encouragement. I am also greatly in debt to Thomas Keller for valuable discussions. Some of the work was done while I was visiting Texas State University-San Marcos. I thank the Mathematics Department for its hospitality.

%%%%%%%%%%%%%%%%%%%%%%%%%%%%%%%%%%%%%%%%%%%%%%%%%%%%%%%%%%%%%%%%%%%%%%%%%

\end{document}